\documentclass[11pt]{amsart}
\usepackage{amsmath,amssymb, xypic,verbatim,eucal, amscd,color,mathrsfs}
\usepackage{mathtools}
\usepackage{graphicx}
\usepackage{colordvi}
\usepackage{mathdots}
\usepackage[colorlinks=true, pdfstartview=FitV,
linkcolor=cyan, citecolor=red, urlcolor=blue]{hyperref}
\usepackage{amsfonts,latexsym,bm}
\usepackage{tikz-cd}
\usepackage{adjustbox}
\usepackage{yhmath}
\usepackage{accents}

\usepackage[all,arc]{xy}
\input xy

\numberwithin{equation}{section}

\theoremstyle{plain}

\newtheorem{theorem}{Theorem}[section]				
\newtheorem{proposition}[theorem]{Proposition}		
\newtheorem{corollary}[theorem]{Corollary}
\newtheorem{lemma}[theorem]{Lemma}

\newtheorem{remark}[theorem]{Remark}

\newtheorem{definition}[theorem]{Definition}

\newcommand{\Ibold}{{\bf I}}

\newcommand{\CBbb}{\mathbb C}

\newcommand{\Bcal}{\mathcal B}
\newcommand{\Ccal}{\mathcal C}

\newcommand{\Ecal}{\mathcal E}
\newcommand{\Fcal}{\mathcal F}

\newcommand{\Hcal}{\mathcal H}
\newcommand{\Ical}{\mathcal I}

\newcommand{\Kcal}{\mathcal K}
\newcommand{\Lcal}{\mathcal L}
\newcommand{\Mcal}{\mathcal M}
\newcommand{\Ncal}{\mathcal N}
\newcommand{\Ocal}{\mathcal O}
\newcommand{\Pcal}{\mathcal P}

\newcommand{\Tcal}{\mathcal T}

\newcommand{\Gfrak}{\mathfrak G}

\newcommand{\Sfrak}{\mathfrak S}

\newcommand{\gfrak}{\mathfrak g}

\newcommand{\sfrak}{\mathfrak s}

\newcommand{\Ascr}{\mathscr A}

\newcommand{\Escr}{\mathscr E}

\newcommand{\Pscr}{\mathscr P}

\newcommand{\SL}{\mathsf{SL}}

\newcommand{\slfrak}{\mathfrak{sl}}

\DeclareMathOperator{\Par}{Par}
\DeclareMathOperator{\Spar}{SPar}

\DeclareMathOperator{\End}{End}

\DeclareMathOperator{\Hom}{Hom}

\DeclareMathOperator{\Aut}{Aut}
\DeclareMathOperator{\aut}{aut}

\DeclareMathOperator{\id}{id}

\DeclareMathOperator{\ad}{ad}
\DeclareMathOperator{\Ad}{Ad}
\DeclareMathOperator{\tr}{tr}

\DeclareMathOperator{\ParAt}{{}^{\emph{par}}At}
\DeclareMathOperator{\SParAt}{{}^{\emph{spar}}At}
\DeclareMathOperator{\ParEnd}{Par}
\DeclareMathOperator{\SParEnd}{SPar}

\newcommand{\WE}{\widetilde {\mathcal{E}}}

\newcommand{\BOXE}{{\mathcal{E}}\boxtimes{\mathcal{E}}'}
\newcommand{\BOX}{\widetilde{\mathcal{E}}\boxtimes\widetilde{\mathcal{E}}'}

\newcommand{\D}{\Delta}

\newcommand{\dbar}{\overline\partial}

\newcommand{\lra}{\longrightarrow}

\newcommand{\At}{{\rm At}}

\newcommand{\isorightarrow}{\xrightarrow{
		\,\smash{\raisebox{-0.5ex}{\ensuremath{\sim}}}\,}}

\begin{document}
	
	\title[A parabolic analog of a theorem of Beilinson and Schechtman]{A parabolic analog of a theorem of Beilinson and Schechtman}

	\author[Biswas]{Indranil Biswas}
\address{Mathematics Department, Shiv Nadar University, NH91, Tehsil Dadri,
	Greater Noida, Uttar Pradesh 201314, India}
\email{indranil.biswas@snu.edu.in, indranil29@gmail.com}
	\author[Mukhopadhyay]{Swarnava Mukhopadhyay}
	\address{School of Mathematics,
		Tata Institute of Fundamental Research,
		Mumbai-400005, India}	\email{swarnava@math.tifr.res.in}
	\author[Wentworth]{Richard Wentworth}
	\address{Department of Mathematics,
		University of Maryland,
		College Park, MD 20742, USA}
	\email{raw@umd.edu}
	
	\thanks{I.B. 
		is	supported in part by a J.C. Bose fellowship and both I.B. and S.M. are also partly supported by DAE, India
		under project no. 1303/3/2019/R\&D/IIDAE/13820. S.M.
		received additional funding from the Science and Engineering Research
		Board, India (SRG/2019/000513). R.W. is 
		supported in part by National Science Foundation grants DMS-1906403
        and DMS-2204346.}
	
	\subjclass[2020]{Primary: 14H60, 32G34, 53D50; Secondary: 81T40, 14F08}

	\begin{abstract}For a simple, simply connected, complex group $G$, 
		we prove an explicit formula to compute the Atiyah class of
        parabolic determinant of cohomology line bundle on the moduli space of
		parabolic $G$-bundles. This generalizes an earlier result of Beilinson-Schechtman.  
	\end{abstract}
	\maketitle
	\allowdisplaybreaks
	
	\thispagestyle{empty}
\section{Introduction}

Lie algebroids play an important role in the geometry of sheaves on
manifolds. For the case of Atiyah algebras associated to principal bundles,
the Atiyah exact sequence packages the information of connections, and more
generally twisted differential operators, on the bundle. Of particular interest are
Atiyah algebras associated to line bundles on moduli spaces of bundles on curves. These are
almost always constructed via descent from a bigger parameter space. A
fundamental question is therefore how to relate the behavior of  Atiyah algebras of these natural line bundles to
infinitesimal joint deformations of the moduli spaces under this
correspondence. 

This question was addressed in the fundamental work of Beilinson-Schechtman
\cite{BS88}. In the context of relative moduli stacks of vector bundles over
families of smooth projective curves, the main result of \cite{BS88}
describes the Atiyah algebra of the determinant of cohomology 
in terms of a direct image of a \emph{trace complex} constructed from the
Atiyah algebra of the universal bundle. This construction is closely related to the
``localization functor'' \cite{BeilinsonBernstein:93,BeilinsonDrinfeld}. 

In \cite{Ginzburg}, Ginzburg gave an alternative construction
that is more amenable to the case of principal bundles. 
This is based on a general correspondence between quasi-Lie algebras and
certain differential  graded Lie algebras.  Applied to the moduli problem, 
this time for principal bundles, 
it can be seen from work of Bloch-Esnault \cite{BE} that the direct image of the
dgla constructed by Ginzburg also computes the Atiyah algebra of the
determinant of cohomology.  

The main goal of this paper is to extend these constructions to the case of
moduli stacks of principal bundles with parabolic structures. 
In order to state the result, let us introduce some notation. Let $\Ccal\,\longrightarrow\,
S$ be a versal family of smooth projective  curves with $n$ marked points
$p_1,\,\cdots,\, p_n$. Fix a simple, simply connected complex algebraic group
$G$ with Lie algebra $\gfrak$. Choose parabolic subgroups $P_1,\,\cdots,\, P_n$
of $G$ and associated weights $\bm\alpha\,=\,(\alpha_1,\,\cdots,\, \alpha_n)$. 
Let $ M_G^{\bm\alpha,rs}\,=\,M_G^{\bm\alpha,rs}(\Ccal/S)\rightarrow S$ be the relative moduli
space (over $S$) of regularly stable parabolic $G$-bundles. 
On $\Ccal\times_S M_G^{\bm\alpha,rs}$ there exists local universal bundles
$\Pcal$. Let ${}^{par}\Sfrak^{\bullet}_{\mathcal{C}\times_S M_{G}^{\bm \alpha ,rs},\pi,S}(\Pcal)$
denote the relative Ginzburg complex associated to $\Pcal$. These local
complexes glue together to give a global complex on 
$\Ccal\times_S M_G^{\bm\alpha,rs}$ (even though the global $\Pcal$ does not exist). 

On the other hand, we consider the sheaf of strongly parabolic relative Atiyah algebras 
${}^{spar}\operatorname{At}_{\Ccal\times_S M_G^{\bm\alpha,rs}/M_G^{\bm \alpha, rs}}(\Pcal)$ defined via the parabolic orbifold correspondence satisfying the short exact sequence
$$0 \rightarrow \Spar(\Pcal)\rightarrow {}^{spar}\operatorname{At}_{\Ccal\times_S M_G^{\bm\alpha,rs}/M_G^{\bm \alpha, rs}}(\Pcal)\rightarrow \mathcal{T}_{\Ccal\times_S M_G^{\bm\alpha,rs}/M_G^{\bm \alpha, rs}}(-D)\rightarrow 0,$$ where $\Spar(\Pcal)$ is the sheaf of strongly parabolic endomorphisms and $D$ is the divisor of marked points. Now consider the sheaf of quasi Lie algebras ${}^{spar}\widetilde{\operatorname{At}}_{\Ccal\times_S M_G^{\bm\alpha,rs}/M_G^{\bm \alpha, rs}}(\Pcal)$ obtained as the pull back of the following sequence 
$$0 \rightarrow {\Omega}_{\Ccal\times_S M_G^{\bm\alpha,rs}/M_G^{\bm \alpha, rs}}\rightarrow ({}^{spar}{\operatorname{At}}_{\Ccal\times_S M_G^{\bm\alpha,rs}/M_G^{\bm \alpha, rs}}(\Pcal)(D))^{\vee}\rightarrow(\Spar(\Pcal)(D))^{\vee} \rightarrow 0,$$
via the isomorphism of sheaf of parabolic endomorphisms $\Par(\Pcal)\cong (\Spar(\Pcal)(D))^{\vee}$. 
Moreover by construction $R^1\pi_{*} {}^{spar}\widetilde{\operatorname{At}}_{\Ccal\times_S M_G^{\bm\alpha,rs}/M_G^{\bm \alpha, rs}}(\Pcal)$ is a sheaf of relative Atiyah algebras on $M_G^{\bm \alpha,rs}$, where $\pi: \mathcal{C}\times_S M_G^{\bm \alpha ,rs}\rightarrow M_G^{\bm \alpha,rs}$ is the projection. We first show that 
$$R^1\pi_{*}({}^{par}\Sfrak^{\bullet}_{\mathcal{C}\times_S M_{G}^{\bm \alpha ,rs},\pi,S}(\Pcal))\simeq R^1\pi_\ast({}^{par}\Sfrak^{-1}_{\mathcal{C}\times_S M_{G}^{\bm \alpha ,rs}/M_G^{\bm \alpha, rs}}(\Pcal)) \simeq  R^1\pi_{*}( {}^{spar}\widetilde{\operatorname{At}}_{\Ccal\times_S M_G^{\bm\alpha,rs}/M_G^{\bm \alpha, rs}}(\Pcal))$$ 
We refere the reader to  Sections 3, 4 for more details. Finally we related it to relative Atiyah algebras of parabolic determinant of cohomologies.

Given a nontrivial holomorphic embedding $\phi\, :\, G\,\longrightarrow\, \SL_r$, there is an
associated determinant of cohomology line bundle $\Lcal_\phi\,\longrightarrow\,
M_G^{\bm\alpha,rs}$. Let $\At_{M_G^{\bm\alpha,rs}/S}(\Lcal_\phi)$
denote the relative  Atiyah algebra of $\Lcal_\phi$. Then the main result
of this paper is the following.

\begin{theorem} \label{thm:main}
On $M_G^{\bm\alpha,rs}$ there is a natural isomorphism of Atiyah algebras
$$
    \frac{1}{m_\phi}\At_{M_G^{\bm\alpha,rs}/S}(\Lcal_\phi)\,\simeq\,
    R^1\pi_\ast({}^{par}\Sfrak^{-1}_{\mathcal{C}\times_S M_{G}^{\bm \alpha ,rs}/M_G^{\bm \alpha, rs}}(\Pcal))\simeq R^1\pi_{*}( {}^{spar}\widetilde{\operatorname{At}}_{\Ccal\times_S M_G^{\bm\alpha,rs}/M_G^{\bm \alpha, rs}}(\Pcal)),
$$
where $m_\phi$ is the Dynkin index of the associated homomorphism
    $\phi_\ast\,:\, \gfrak\,\to\,\slfrak_r$ given by the ration of the normalized Killing forms. 
\end{theorem}

In a recent paper \cite{BBMP20} the result of
Beilinson-Schechtman was used in an integral way to give an algebraic proof
of the existence of a flat projective connection (a \emph{Hitchin
connection}) on the bundle of
generalized theta functions for vector bundles on families of curves. One
of the main motivations of the present paper was to apply Theorem \ref{thm:main}
in the same way to obtain  a Hitchin connection for theta functions
associated to parabolic $G$-bundles. This is carried out in \cite{BMW1}.

\section{Quasi-Lie algebras and extensions of Atiyah algebras}
\label{sec:setup}

\subsection{Basic definitions} \label{sec:ginzburg}

In this section, we recall a correspondence stated in Ginzburg \cite{Ginzburg} between quasi-Lie algebras and their associated 
differential graded Lie algebras. We also recall from Beilinson-Schechtman \cite{BS88} a natural classes of Atiyah algebras associated 
to a family of curves.

\subsubsection{Quasi-Lie algebras} \label{sec:dgla}

First we recall the definition of a quasi-Lie algebra. Let
$\widetilde\gfrak$
be a vector space equipped with a skew-symmetric bilinear map $$[\ ,\
]\,:\, \widetilde\gfrak \times \widetilde\gfrak \,\longrightarrow\,
\widetilde\gfrak .$$ Let $Z\,\subset\, \widetilde{\mathfrak{g}}$ be a linear subspace.  

\begin{definition}
A triple $(\widetilde{\mathfrak{g}},\, Z,\, [ \ , \ ])$ as above is called a quasi-Lie algebra if the following holds: 
	\begin{enumerate}
		\item the subspace $Z$ is central with respect to  $[\ , \ ]$, and
		\item the bracket $[\ , \ ]$ descends 	to give a Lie algebra structure on $\widetilde{\mathfrak{g}}/Z$. 
	\end{enumerate}
\end{definition}
We will see that in the setting of Atiyah algebras, these quasi-Lie algebras arise naturally. 
Now recall the notion of a differential graded Lie algebra (dgla).

\begin{definition}
A \emph{differential graded Lie algebra} (dgla)
is a vector space $\mathfrak{S}\,:=\, \bigoplus_i \mathfrak{S}^i$ together with a bilinear map
$\{\mathfrak{S}^i,\mathfrak{S}^{j}\} \,\subset\, \mathfrak{S}^{i+j}$ and a differential
$d\,:\, \mathfrak{S}^{i}\,\longrightarrow\, \mathfrak{S}^{i+1}$ satisfying the following:
\begin{itemize}
\item $\{x,\,y\}\,=\,(-1)^{|x||y|+1}\{y,\,x\}$, where $|z|\,=\, i$ for $z\,\in\, \mathfrak{S}^i$,

\item $(-1)^{|x||z|}\{x,\,\{y,\,z\}\}+ (-1)^{|y||x|}\{y,\,\{z,\,x\}\}+
(-1)^{|y||z|}\{z,\,\{y,\,x\}\}\,=\,0$, and

\item $d\{x,\,y\}\,=\,\{dx,\,y\}+ (-1)^{|x|}\{x,\,dy\}$.
\end{itemize}
\end{definition}

A morphism of dglas
is a graded linear map $\mathfrak{S}\,\longrightarrow\, \mathfrak{S}'$ that preserves the Lie bracket and commutes with the differentials. 

The following lemma of  Ginzburg \cite[Lemma 7.7]{Ginzburg} gives a correspondence between quasi-Lie algebras and a certain class of dglas. 

\begin{lemma}\label{lem:ginzburg}
Let $(\widetilde{\mathfrak{g}},\, Z,\, [ \ , \ ])$ be a quasi-Lie
algebra equipped with a symmetric $Z$-valued, $Z$-invariant bilinear form $\langle \ ,\ \rangle \,:\, \operatorname{Sym}^2
\widetilde{\mathfrak{g}}\,\longrightarrow\, Z$ such that the following hold: 
\begin{itemize}
\item $\langle [x,\,y],\,z\rangle +\langle y,\, [x,\,z]\rangle\,=\,0$, and

\item $[x,\,[y,\,z]]+ [y,\,[z,\,x]]+ [z,\,[x,\,y]]\,=\,d(\langle [x,\,y],\,z\rangle )$.
\end{itemize}
Then there exists a dgla $\mathfrak{S}\,=\,\mathfrak{S}^{-2}\oplus
\mathfrak{S}^{-1}\oplus \mathfrak{S}^{0}$, where
$\Sfrak^0 := \gfrak$, 
$\mathfrak{S}^{-1}\,:=\,\widetilde{\mathfrak{g}}$,\, $\mathfrak{S}^{-2}\,:=\,Z$,
the differential is given by inclusion and quotient, 
and with the bracket given by the formula 
$\{x,\,y \}\,=\,\langle x,\,y\rangle$; and $\{x,\,dy\} \,=\,[x,\,y]+\langle x,\, y
\rangle$  for $x,\,y \,\in\, \widetilde{\mathfrak{g}}$. 
    Conversely given a dgla as above satisfying $\{\Sfrak^{-2},
    d\Sfrak^{-1}\}\subset \ker d$, 
    there exists a quasi-Lie algebra along with a symmetric $Z$-valued,
    $Z$-invariant bilinear form. 
\end{lemma}

\subsubsection{Dglas associated to families of curves}\label{sec:families}

Let $\pi\,:\,X\rightarrow T$ be a smooth morphism of relative dimension one parametrized by $T$;
the curves in this family are not assumed to be proper. The relative holomorphic tangent bundle $\mathcal{T}_{X/T}$
fits in the exact sequence of $\mathcal{O}_X$-modules 
$$0\,\lra\,\mathcal{T}_{X/T}\,\lra\,\mathcal{T}_X \,\stackrel{d\pi}{\lra}\,\pi^* \mathcal{T}_{T} \,\lra\, 0.$$
Let $\mathcal{T}_{X,\pi}\, \subset\, \mathcal{T}_X$ denote the subsheaf $d\pi^{-1}(\pi^{-1}\mathcal{T}_T)$. Clearly
the sheaf $\mathcal{T}_{X,\pi}$ has the structure of Lie algebra
with Lie bracket coming from that on $\mathcal{T}_X$ and there is an exact sequence of Lie algebras 
$$0\,\lra\,\mathcal{T}_{X/T}\,\lra\,\mathcal{T}_{X,\pi} \,\stackrel{d\pi}{\lra}\,
\pi^{-1}\mathcal{T}_{T}\,\lra\, 0.$$ 
Consider the dgla given by
$\mathcal{T}_{\pi}^{\bullet}\,=\,\bigoplus_i \mathcal{T}^i,$ where $\mathcal{T}^i$ is zero for $i\,\neq\, \{0,\,-1\}$,
$\mathcal{T}^{-1}\,:=\,\mathcal{T}_{X/T}$ and $\mathcal{T}^{0}\,:=\,\mathcal{T}_{X,\pi}$. 
This dgla $\mathcal{T}_{\pi}^{\bullet}$ carries a natural action of $\pi^{-1}\mathcal{O}_T$ and a map
\begin{equation}\label{z1}
\overline{\epsilon}\,:\, \mathcal{T}_{\pi}^{\bullet}\,\longrightarrow\, H^0(\mathcal{T}_{\pi}^{\bullet})
\,=\,\pi^{-1}(\mathcal{T}_{T})
\end{equation}
given by $d\pi$. The relative de Rham complex $\Omega^{\bullet}_{X/T}\,=\,(\mathcal{O}_X\rightarrow \omega_{X/T})$
with $V^0\,:=\,\mathcal{O}_X$ and $V^1\,:=\,\Omega_{X/T}$ is naturally a $dg$-module $V^{\bullet}\,:=\,V^0\oplus V^1$
for $\mathcal{T}^{\bullet}_{\pi}$ which is compatible with the $\pi^{-1}\mathcal{O}_T$ action on both sides. 

\subsubsection{Atiyah algebras as $R^0\pi_*$ of dglas} \label{sec:atiyah}

Let $\pi\,:\,X\,\longrightarrow\, T$ be a family of curves as before. We discuss the notion of $\pi$-algebras following Beilinson-Scechtman \cite{BS88} which are
quasi-isomorphic to extensions of a complex of Atiyah algebras by the de Rham complex $\Omega^{\bullet}_{X/T}$.

\begin{definition}[{\cite[\S~1.2.1]{BeilinsonSchechtman:88}}]
An $\mathcal{O}_S$-Lie algebra $\mathcal{A}^{\bullet}$ on $X$ is a dgla together with a $\pi^{-1}\mathcal{O}_S$-module structure
and a morphism $\epsilon_{\mathcal{A}}\,:\, \mathcal{A}^{\bullet}\,\longrightarrow\,\mathcal{T}^{\bullet}_{\pi}$ that satisfies 
the condition $[a,\,fb]\,=\,\overline{\epsilon}_{\mathcal{A}}(a)(f)b + f[a,b]$, where
$\overline{\epsilon}_{\mathcal{A}}\,=\,\overline{\epsilon}\circ \epsilon_{\mathcal{A}}$
(see \eqref{z1}). A $\pi$-algebra $\mathcal{A}^{\bullet}$ is an $\mathcal{O}_S$-Lie algebra together with a three term filtration 
$$
0\,=\,\mathcal{A}^{\bullet}_{-3} \,\subset\, \mathcal{A}^{\bullet}_{-2} \,\subset\, \mathcal{A}^{\bullet}_{-1}
\,\subset\, \mathcal{A}^{\bullet}_{0}\,=\,\mathcal{A}^{\bullet}
$$
such that the following hold:
\begin{enumerate}
\item $[\mathcal{A}^{\bullet}_i, \mathcal{A}^{\bullet}_j]\,\subset\, \mathcal{A}^{\bullet}_{i+j}$,\,
$\mathcal{O}_S \cdot\mathcal{A}^{\bullet}_i \,\subset\, \mathcal{A}^{\bullet}_i$,

\item $\mathcal{A}^{\bullet}_{-2}\,\simeq\, \Omega_{X/S}^{\bullet}[2]$ as $\mathcal{O}_S$-modules,

\item $\mathcal{A}^{\bullet}_{-1}/\mathcal{A}^{\bullet}_{-2}$ is acyclic,

\item $\epsilon_{\mathcal{A}}\,:\,\mathcal{A}^{\bullet}\,\longrightarrow\,
\mathcal{A}^{\bullet}/\mathcal{A}^{\bullet}_{-1}\,\simeq\, \mathcal{T}^{\bullet}_{\pi}$, and

\item the ad action of $\mathcal{A}^{\bullet}/\mathcal{A}^{\bullet}_{-1}$ on
$\mathcal{A}^{\bullet}_{-2}$ coincides with the $ \mathcal{T}^{\bullet}_{\pi}$-action on $\Omega^{\bullet}_{X/T}[2]$. 
\end{enumerate}
\end{definition}

By the above definition of a $\pi$-algebra, we get an exact sequence 
\begin{equation}
0\,\lra\,\mathcal{A}^{\bullet}_{-2}\,\lra\,\mathcal{A}^{\bullet}
\,\lra\,\mathcal{A}^{\bullet}/\mathcal{A}^{\bullet}_{-2}\,\lra\,0.
\end{equation}
This will be called an $\Omega$-extension of $\mathcal{A}^{\bullet}/\mathcal{A}^{\bullet}_{-2}$. Now assume
that the map $\pi$ is proper. Suppose we are given any $\pi$-algebra $\mathcal{A}^{\bullet}$ fitting
in an exact sequence of complexes
\begin{equation}\label{e1}
0 \,\lra\,\Omega_{X/T}^\bullet[2]\,\lra\,\mathcal{A}^{\bullet}
\,\lra\,\mathcal{A}^{\bullet}/\mathcal{A}^{\bullet}_{-2}\,\lra\, 0.
\end{equation}

\begin{proposition}[{\cite[\S~1.2.3]{BeilinsonSchechtman:88}}]
The short exact sequence
$$0\,\lra\,\mathcal{O}_T \,\lra\, R^0\pi_{*}\mathcal{A}^{\bullet}\,\lra\,\mathcal{T}_T
\,\lra\, 0$$ 
defines an Atiyah algebra on $T$.
\end{proposition}

\subsubsection{$\pi$-algebras associated to Atiyah algebras} \label{sec:pi}

Let $\pi\,:\, X\, \longrightarrow\, T$ be a family of curves that are not necessarily projective, and let $\mathcal{A}$ be an
$R$-Atiyah algebra on $X$. There is a natural $\pi$-algebra associated to $\mathcal{A}$. Consider the
$\mathcal{O}_T$-Lie algebra $\mathcal{A}^{\bullet}_{\pi}$ defined by: 
$\mathcal{A}_{\pi}^{-1}\,:=\,\epsilon_{\mathcal{A}}^{-1} \mathcal{T}_{X/T}$ and
$\mathcal{A}^0_{\pi}\,:=\,\epsilon_{\mathcal{A}}^{-1} \mathcal{T}_{X,\pi}$.
There is a canonical surjective map $\epsilon_{\mathcal{A}}\,:\, \mathcal{A}_{\pi}^{\bullet}\,\longrightarrow
\, \mathcal{T}_{\pi}^{\bullet}$ whose kernel is $\operatorname{Cone}\operatorname{id}_{R}.$ 

\begin{definition}
A $\Omega$-extension ${}^{\#}\mathcal{A}^{\bullet}$ of $(\mathcal{A},\,R)$ is a $\Omega$ extension of
$\mathcal{A}_{\pi}^{\bullet}$ together with an $\mathcal{O}_X$-module structure on ${}^{\#}\mathcal{A}^{-1}$ such that 
\begin{itemize}
\item the $\mathcal{O}_X$-action is compatible with the action on $\mathcal{A}_{\pi}^{-1}$, and

\item the component $[\ ,]_{-1,-1}\,:\, {}^{\#}\mathcal{A}^{-1}\otimes {}^{\#}\mathcal{A}^{-1}\,\longrightarrow\,
{}^{\#}\mathcal{A}^{-2}\,=\,\mathcal{O}_X$ is a differential operator along the fibers. 
\end{itemize}
\end{definition}
There is a commutative diagram 
\begin{equation}\label{eqn:defdiag}
\begin{tikzcd}
& 0\ar[d] &&&\\
0\ar[r]& \mathcal{O}_X \ar[r, equal]\ar[d] & {}^{\#}\mathcal{A}^{-2}\ar[d]\ar[r] &0\ar[d]&\\
0 \ar[r] & \Omega_{X/T} \ar[d]\ar[r] & {}^{\#}\mathcal{A}^{-1} \ar[r]\ar[d]    &\mathcal{A}_{\pi}^{-1}:=\mathcal{A}_{X/T}\ar[r]\ar[d]&0 \\
&0 \ar[r] & {}^{\#}\mathcal{A}^{0}\ar[r,equal ] &\mathcal{A}_{\pi}^{0}:=\mathcal{A}_{X,\pi} \ar[r] & 0.
\end{tikzcd}
\end{equation}
Observe that the $\Omega_{X/T}$-extension is, by definition, a $\pi$-algebra, 
where the filtration ${}^{\#}\mathcal{A}^{\bullet}_{-2}$ is given by $\Omega_{X/T}^{\bullet}[2]$.

Later will also need to vary $T$ with respect to $S$ and consider $\mathcal{T}_{\pi,S}^{\bullet}$ where degree $-1$ term is same as $\mathcal{T}_{\pi}^{\bullet}$ and the degree zero term is $\mathcal{T}_{X,\pi,S}\,:=\,d\pi^{-1}(\pi^{-1}(\mathcal{T}_{T/S}))$. In the relative set-up one similarly defines $\mathcal{A}_{X,\pi,S}^{\bullet}$ by modifying the zero-th term to be $\mathcal{A}_{X,\pi,S}^{-0}:=\epsilon_{\mathcal{A}}^{-1}\mathcal{T}_{X,\pi,S}$. The resulting pushforward $R^0\pi_{*}\mathcal{A}_{X,\pi,S}^{\bullet}$ is a relative Atiyah algebra satisfying the fundamental exact sequence
$$0\rightarrow \mathcal{O}_T \rightarrow R^0\pi_*(\mathcal{A}_{X,\pi,S}^{\bullet})\rightarrow \mathcal{T}_{T/S}\rightarrow 0$$

\subsection{Principal bundles}\label{sec:principal}

\subsubsection{The Ginzburg complex} \label{sec:g-complex}

We continue with the earlier notation.
Let $G$ be a complex simple Lie group with Lie algebra $\gfrak$. We will denote by $\kappa_{\gfrak}$ the normalized
Cartan-Killing form on $\gfrak$ and consider the corresponding isomorphism
\begin{equation}\label{e2}
\nu_{\gfrak}\,:\, \gfrak \,\xrightarrow{\ \simeq\ }\,\gfrak^{\vee}
\end{equation}

Let $\Pi \,:\, P\,\longrightarrow\, X$ be a holomorphic principal $G$ bundle; we use the
convention that $G$ acts on the right of $P$. Automorphisms $\Aut(P)\vert_U$ of $P$ over
$U\,\subset\, X$ are by definition $G$-equivariant automorphisms of
$\Pi^{-1}(U)$, i.e.\ $F\,:\,P\vert_U\,\longrightarrow\, P\vert_U$ satisfying $F(pg)\,=\,F(p)g$ for all $g\,\in\, G$;
we do not assume that $\Pi\circ F\,=\, \Pi$. The group of automorphisms $\Aut(P)\vert_U$ is generated by the
invariant vector fields
$$
\aut_U(P)\,:=\,\left\{Y\,\in\, \Gamma(\Pi^{-1}(U),\, TP)\mid (R_g)_\ast
Y\,=\,Y\, ,\ \ \forall\,\  g\,\in\, G\right\}.
$$ 
Then $\aut(P)$ defines a coherent sheaf of $\Ocal_X$-modules.
We refer to the subsheaf $\At_{X/T}(P)\, \subset\, \aut(P)$ that projects by $d\Pi$ to
$\Tcal_{X/T}\, \subset\, TX$ as the \emph{relative Atiyah algebra of $P$}. 
We have an exact sequence
\begin{equation}\label{eqn:atiyahG}
0\,\lra\, \ad(P)\,\lra\, \At_{X/T}(P)\,\lra\, \Tcal_{X/T} \,\lra\, 0.
\end{equation}
We will explain the inclusion map on the left. Recall that a section of
$\ad(P)$ is identified with a function $f\,:\,P\,\longrightarrow\, \gfrak$
satisfying $f(pg)\,=\,\Ad_{g^{-1}}f(p)$. For $Y\,\in\,\gfrak$, let
$Y^\sharp$ denote the fundamental vector field on $P$
generated by $Y$. Then $Y^\sharp(pg)\,=\,(R_g)_\ast(\Ad_g Y)^\sharp(p)$. 
The map $\ad(P)\,\longrightarrow \,\aut(P)$ in \eqref{eqn:atiyahG} is
$f\,\longmapsto\, Y$, where $Y(p)\,=\,f(p)^\sharp$. With this definition,
$$
Y(pg)\,=\,(\Ad_{g^{-1}}f(p))^\sharp(pg)\,=\,(R_g)_\ast
(f(p))^\sharp\,=\,Y(p),
$$
so $Y$ is invariant and lies in $\aut(P)$ as the kernel of
$d\Pi$. The following will be important when we investigate universal bundles.

\begin{remark} \label{rem:adjoint}
Let $Z(G)$ denote the (finite) center of $G$, and let $\overline P\,:=\,P/Z(G)$ 
the associated principal bundle for the adjoint group $\overline G\,:=\, G/Z(G)$. 
Then there are canonical isomorphisms $\ad(\overline P)\,\simeq\,
\ad(P)$ and $\At_{X/T}(\overline P)\,\simeq \,\At_{X/T}(P)$.
\end{remark}

Dualizing \eqref{eqn:atiyahG} gives a quasi-Lie algebra structure on $\At_{X/T}(P)^{\vee}$: 
\begin{equation} \label{eqn:dualatiyahG}
0 \,\lra\, \Omega_{X/T} \,\lra\,\At_{X/T}(P)^{\vee} \,\stackrel{j}{\lra}\,\ad(P)^{\vee} \,\lra\, 0.
\end{equation}
Identify $\ad(P)$ with $\ad(P)^{\vee}$ using $\nu_{\gfrak}$ in \eqref{e2}, and denote
$\At_{X/T}(P)^{\vee}$ by $\widetilde{ \mathfrak{g}}_P$.
We have the following quasi-Lie algebra
\begin{equation}
\begin{tikzcd}
0\arrow[r] & \Omega_{X/T}\arrow[r]\arrow[d,equal] & \At_{X/T}(P)^{\vee} \arrow[r]& \ad(P)^{\vee} \arrow[r]&0 \\
0 \arrow[r]& \Omega_{X/T} \arrow[r] & \widetilde{ \mathfrak{g}}_P \arrow[u]\arrow[r]& \ad(P)\arrow[u, "\cong", "\nu_{\gfrak}"'] \arrow[r]\arrow[r]&0.
\end{tikzcd}
\end{equation}
As in Lemma \ref{lem:ginzburg}, associated to $\widetilde\gfrak_P$ is a dgla,
$\Sfrak^i_{X/T}(P)$, which we call the \emph{Ginzburg complex} for $P$.  
Explicitly, 
$$
\Sfrak^i_{X/T}(P)=\begin{cases} 0 & i\neq -2, -1, 0 \\
\Ocal_X & i=-2 \\
\widetilde\gfrak_P & i=-1 \\
\gfrak_P& i=0.\end{cases}
$$
We will later consider a relative version of it where $T$ varies.

\subsubsection{The Bloch-Esnault complex}\label{sec:be-complex}

Let $\Ecal\,\longrightarrow\, X$ be a holomorphic vector bundle. 
Like the Ginzburg complex, the Bloch-Esnault complex $\Bcal_{X/T}^\bullet(\Ecal)$, \cite{BE},
is nonzero in degrees $-2, -1$, and $0$, with $\Bcal_{X/T}^{-2}(\Ecal)\,=\,\Ocal_X$.
To define the other terms, let $X\,\simeq\, \Delta\,\subset\, X\times_T X$ be the relative diagonal, and   
let $\At_{X/T}(\Ecal)$ denote the relative Atiyah algebra of $X\,\longrightarrow\, T$. 
Then $\Bcal_{X/T}^0(\Ecal)\,=\,\operatorname{End}(\mathcal{E})$.
Set $\Ecal'\,:=\,\Ecal\otimes \Omega_{X/T}$, and define
$$\widetilde\Bcal_{X/T}^{-1}(\Ecal)\,=\,\frac{ \Ecal\boxtimes \Ecal'
(\Delta)}{\Ecal\boxtimes \Ecal' (-\Delta)}
$$
on $X\times_T X$. 
Note that $\Ocal_{X\times_T X}(\Delta)\big\vert_\Delta\,=\,\Tcal_{X/T}$. 
Then $\Bcal^{-1}_{X/T}(\Ecal)$ is defined by pushing out with the trace:
$$
\begin{tikzcd}
0 \arrow{r} &
\displaystyle
\frac{ \Ecal\boxtimes \Ecal'
}{\Ecal\boxtimes \Ecal' (-\Delta)}
\arrow{r} \arrow[d, equal]
&
\displaystyle
\frac{ \Ecal\boxtimes \Ecal'
	(\Delta)}{\Ecal\boxtimes \Ecal' (-\Delta)}
\arrow{r} \arrow[d, equal]
&
\displaystyle
\frac{ \Ecal\boxtimes \Ecal'
	(\Delta)}{\Ecal\boxtimes \Ecal'}
\arrow{r} \arrow[d, equal]& 0 
\\
0 \arrow{r} &  \End(\Ecal)\otimes \Omega_{X/T} \arrow{r} \arrow["\tr"]{d} &
\widetilde\Bcal^{-1}_{X/T}(\Ecal)  \arrow{r} \arrow{d}&
\End(\Ecal)\arrow{r}\arrow[equal]{d} & 0 \\
0 \arrow{r} &  \Omega_{X/T}\arrow{r} & \Bcal^{-1}_{X/T}(\Ecal) \arrow{r}
&
\End(\Ecal) \arrow{r}& 0.
\end{tikzcd}
$$
Here $\tr$ denote the trace $\operatorname{End}(\Ecal)\,\longrightarrow\, \mathcal{O}_X$ of endomorphisms.

We will actually need the traceless version $\Bcal_{0,X/T}^\bullet(\Ecal)$, where
$\Bcal_{0,X/T}^{-2}(\Ecal)\,=\,\Bcal^{-2}_{X/T}(\Ecal)$ while $\Bcal_{0,X/T}^0(\Ecal)$ is defined by
the inclusion map into the traceless relative Atiyah algebra 
and $\Bcal_0^{-1}(\Ecal)$ is defined by pulling back the extension
$\Bcal^{-1}_{X/T}(\Ecal)$ over the sheaf $\End_0(\Ecal)$ of traceless endomorphisms.

\subsubsection{Associated bundles} \label{sec:associated}

For any simple Lie algebra $\gfrak$, recall $\nu_{\gfrak}$ in \eqref{e2}.

\begin{lemma}\label{lemma:morningfirst}
Let $\phi_\ast\,:\, \gfrak\,\longrightarrow\, \sfrak$ be
a nonzero homomorphism of simple Lie algebras.
Consider the linear map $\underline\psi\,:\, \sfrak \,\longrightarrow\, \gfrak$ 
given by the following diagram:
\begin{equation}
\begin{tikzcd}
&	\sfrak^\vee \arrow[r, "\phi_\ast^t"] &
\gfrak^\vee\arrow[d,"\cong", "\nu_\gfrak^{-1}"']&\\
\gfrak\arrow[r,"\phi_\ast"]
&	\sfrak \arrow[r,"\underline{\psi}"]\arrow[u,
"\cong","\nu_\sfrak"']& \gfrak. 
\end{tikzcd}
\end{equation}
Then $\underline{\psi}\circ \phi_\ast\,=\,m_{\phi}
\operatorname{Id}_{\gfrak}$, where as mentioned in the Introduction, $m_{\phi}$ is the Dynkin index
of $\phi_\ast$. Moreover, replacing $\underline{\psi}$ by
$\psi\,:=\,(m_{\phi})^{-1}\underline{\psi}$,
$${\psi}\circ \phi_\ast\,=\,\operatorname{Id}_{\gfrak} .$$
\end{lemma}

\begin{proof}
This follows form a direct calculation and the definition of Dynkin index.
\end{proof}

The following lemma is straightforward.

\begin{lemma}
Let ${\psi}\,=\,\frac{1}{m_{\phi}}(\nu_\gfrak^{-1}\circ\phi_*^t
\circ \nu_\sfrak)$ be as in Lemma \ref{lemma:morningfirst}, and consider the map 
$\varphi\,:\,\gfrak \,\longrightarrow\, \sfrak$ 
defined by $\nu_{\mathfrak{s}}^{-1}\circ \psi^{t}\circ \nu_{\gfrak}$:
\begin{equation}
\begin{tikzcd}
\gfrak\arrow[r, "\nu_\gfrak" ]\arrow[rrr, bend right=35,
"\varphi", swap ]
&  \gfrak^\vee
\arrow[r,"{\psi}^t"] &\sfrak^\vee \arrow[r, "\nu_\sfrak^{-1}" 
]& \sfrak.
\end{tikzcd}
\end{equation} 
Then $\varphi\,=\,\frac{1}{m_{\phi}}\phi_\ast$.
\end{lemma}

Let
$\phi\,:\, G\,\longrightarrow\, \SL(V)$ be a nontrivial holomorphic representation and 
$\Ecal_\phi\,=\,P\times_G V$ the corresponding vector bundle associated to a principal $G$-bundle
$\Pi \,:\, P\,\longrightarrow\, X$. Sections of
$\Ecal_\phi$ are functions $\sigma \,:\, P\,\longrightarrow\, V$ satisfying
the condition that $\sigma(pg)\,=\,\phi(g^{-1})\sigma(p)$ for all $g\, \in\, G$.
The adjoint bundle $\ad(P)\,=\,P\times_G\gfrak$ maps to the traceless
endomorphism bundle
$\End_0(\Ecal_\phi)$ using $\phi_\ast\,:=\,d\phi\,:\,\gfrak\,\longrightarrow\, \slfrak(V)$,
and we shall use the same notation $\phi_\ast$ for this map $\ad(P)\,\longrightarrow\, \End_0(\Ecal_\phi)$.
Notice that in this case, the map $\psi$ defined in Lemma
\ref{lemma:morningfirst} is $G$-equivariant, and hence defines a homomorphism 
$\End_0(\Ecal_\phi)\,\longrightarrow\, \ad(P)$. 
A $G$-invariant vector field $Y$ on $P$
defines a differential operator on sections by $Y(\sigma)$
($=\,d\sigma(Y)$). Invariance of $Y$ guarantees that $Y(\sigma)$
is again equivariant  with respect to $\phi$, and so defines a
section of $\Ecal_\phi$.  It is clear that the symbol of this
operator is $\Pi_\ast Y$. Therefore, denoting the relative
Atiyah algebra of $\Ecal_\phi\,\longrightarrow\, X$ by $\At_{X/T}(\Ecal_\phi)$, we have
constructed a map (cf.\ Atiyah \cite[Page 188]{Atiyahconn})
\begin{equation}\label{eqn:atiyah}
\widetilde\phi\,\,: \,\,\At_{X/T}(P)\,\,\longrightarrow\,\, \At_{X/T}(\Ecal_\phi) .
\end{equation}

The following is a consequence of the above.

\begin{proposition}\label{prop:wrongwaymap}
For a principal $G$-bundle $P\,\longrightarrow\, X$, a
representation $\phi\,:\,G\,\longrightarrow \,\SL(V)$, and the
associated vector bundle $\Ecal_\phi\,\longrightarrow\, X$, there is a natural map
$\widetilde\psi\,:\, \At_{X/T}(\Ecal_\phi)\,\longrightarrow\, \At_{X/T}(P)$ that makes the following
diagram commutative:
\begin{equation}
\begin{tikzcd}
0 \arrow[r] & \ad(P)\arrow[r] \arrow[d, "\phi"]& \At_{X/T}(P)\arrow[r]
\arrow[d,  "\widetilde{\phi}"] & \mathcal{T}_{X/T} \arrow[r]\arrow[d, equal]&0\\
0\arrow[r]& \End_0(\Ecal_\phi)\arrow[d, "{\psi}"]
\arrow[r]& \At_{X/T}(\Ecal_\phi)
\arrow[r]\arrow[d, "\widetilde{\psi}"]&\mathcal{T}_{X/T}
\arrow[r]\arrow[d, equal]& 0  \\
0 \arrow[r] & \ad(P) \arrow[r] \arrow [from=uu, bend left=75, crossing
over, "\operatorname{Id}" description, near start ]&
\At_{X/T}(P) 
\arrow[r] \arrow [from=uu, bend left=75, crossing over, "\operatorname{Id}" description , near start ]& \mathcal{T}_{X/T}  \arrow[r]& 0.
\end{tikzcd}
\end{equation}
Here $\psi\,:\, \End_0(\Ecal_\phi) \,\longrightarrow\, \ad(P)$ is induced from 
the map $\psi$ constructed in Lemma \ref{lemma:morningfirst} and $\widetilde{\psi}$ is
the map obtained by pushing out the exact sequence of $At_{X/T}(\mathcal{E}_{\phi})$ via $\psi$.
\end{proposition}

\subsubsection{Relating the Ginzburg and Bloch-Esnault complexes}\label{sec:comparison}

We now compare the Ginzburg complex for a principal bundle $P$ with
the Bloch-Esnault  complex for the bundle $\Ecal_\phi$ associated to
$P$ via a nontrivial representation $\phi\,:\,G\,\longrightarrow\, \SL(V)$. 
Identifying the endomorphism bundle $\End_0(\Ecal_\phi)$ with its dual using the trace
homomorphism, consider the map 
$$
\phi_\ast^t \,:\,\End_0(\Ecal_\phi)\,\cong\, (\End_0(\Ecal_\phi))^{\vee}\,\lra
\,\ad(P)^\vee\ .$$
Let $\Kcal_\phi\,=\,\ker\phi_\ast^t$. The following is important for us.

\begin{proposition}\label{prop:tricky1} There is a lift of the inclusion
$\Kcal_\phi\, \longrightarrow\, \End_0(\Ecal_\phi)$:
	\begin{equation}
	\begin{tikzcd}
	&\Bcal_{0,X/T}^{-1}(\Ecal_\phi)\arrow[d, two heads] \\
	\Kcal_\phi \arrow[r] \arrow[ur, dashed] & \End_0(\Ecal_\phi)
	\end{tikzcd}
	\end{equation}
\end{proposition}

\begin{proof}
We begin by describing a general situation. Namely, for bundles $\Ecal$,
$\Fcal$, we find a local lifting $\Phi_\alpha$ of the map
$$
\frac{ \Ecal\boxtimes \Fcal'
(\Delta)}{\Ecal\boxtimes \Fcal' (-\Delta)}
\,\lra\, \frac{ \Ecal\boxtimes \Fcal'
(\Delta)}{\Ecal\boxtimes \Fcal' } \,\simeq\, \Hom(\Fcal,\,\Ecal)
\,\lra\, 0.
$$
Choose local trivializations of $X\,\longrightarrow\, T$ and coordinate 
neighborhoods $U_\alpha,\, z_\alpha$ on a fixed fiber.
We set $\varphi_{\alpha\beta}\,=\,z_\alpha\circ z_\beta^{-1}$, so 
$\varphi_{\alpha\beta}'\,=\,dz_\alpha/dz_\beta$. The lift is
defined by choosing a (holomorphic) connection $\nabla_\alpha$ on
$\Fcal^\ast$.

Given sections $u$ and $v$ of $\Ecal$ and $\Fcal^\ast$ respectively on $U_\alpha$,  
define on $U_\alpha\times U_\alpha$
\begin{equation} \label{eqn:phi} 
\Phi_\alpha(u(q)\otimes v(q))\,=\,
\frac{u(p)\boxtimes
v(q)dz_\alpha(q)}{z_\alpha(p)-z_\alpha(q)}+u(p)\boxtimes\nabla_\alpha
v(q)\, \ \mod\,\ \Ical_\Delta.
\end{equation}
To show that this is well-defined, let $f(p,\,q)\,\in\, \Ical_\Delta$.
Multiplying on the right hand side, we have
\begin{equation}\label{eqn:error} 
\left(
\frac{f(p,q)dz_\alpha(q)}{z_\alpha(p)-z_\alpha(q)}+\frac{\partial
		f}{\partial q}(p,q)\right)u(p)\boxtimes v(q) \ \mod\ \Ical_\Delta.
	\end{equation} 
	Since 
	$$
	f(p,q)=\frac{\partial f}{\partial q}(p,p)(z_\alpha(q)-z_\alpha(p))
	\ \mod\ \Ical_\Delta^2,
	$$
	we see that \eqref{eqn:error} vanishes, and so \eqref{eqn:phi} gives a
	well-defined lift. 
	Set
$\Ascr_{\alpha\beta}\,=\,\nabla_\alpha-\nabla_\beta$ and
$\Theta_{\alpha\beta}\,=\,(\varphi''_{\alpha\beta}/\varphi'_{\alpha\beta})dz_\beta$.
Then $\Ascr_{\alpha\beta}\,\in\, \End(\Fcal^\ast)\otimes \Omega_{X/T}$
is a 1-cocycle representing the Atiyah class of $\Fcal^\ast$, and $\Theta_{\alpha\beta}\,\in\, \Omega_{X/T}$ 
is a cocycle for the affine structure (cf.\ \cite[p.\ 164]{Gunning66}). 
	
Notice that
\begin{align*}
z_\alpha(p)-z_\alpha(q)
&\,=\,\varphi_{\alpha\beta}(z_\beta(p))-\varphi_{\alpha\beta}(z_\beta(q))
\\
&=\,\varphi'_{\alpha\beta}(z_\beta(q))(z_\beta(p)-z_\beta(q))+\frac{1}{2}
\varphi''_{\alpha\beta}(z_\beta(q))(z_\beta(p)-z_\beta(q))^2+\cdots
\end{align*}
from which we have
$$
\frac{dz_\alpha(q)}{z_\alpha(p)-z_\alpha(q)}\,=\,
	\frac{dz_\beta(q)}{z_\beta(p)-z_\beta(q)}-\frac{1}{2}\Theta_{\alpha\beta}
	\ \ \mod\,\ \Ical_\Delta.
	$$
	The cocycle 
	$\Phi_{\alpha\beta}\,=\,\Phi_\alpha-\Phi_\beta\,\in\,
	\Hom\left(\Hom(\Fcal,\,\Ecal),\, \Hom(\Fcal,\,\Ecal)\otimes
	\Omega_{X/T}\right)
	$
	is then given by
	$$
	\Phi_{\alpha\beta}(u\otimes v)\,=\, u\otimes
	\Ascr_{\alpha\beta}v-\frac{1}{2}u\otimes v\otimes\Theta_{\alpha\beta}.
	$$
In the  case of $\Fcal\,=\,\Ecal$, we may write simply
	\begin{equation*} \label{eqn:cocycle}
	\Phi_{\alpha\beta}\,:\, \End(\Ecal)\,\lra\, \End(\Ecal)\otimes
	\Omega_{X/T}\,:\, \beta\,\longmapsto\,
	\beta\left(\Ascr_{\alpha\beta}-\frac{1}{2}\Theta_{\alpha\beta}\cdot\Ibold\right).
	\end{equation*}
	The extension class for $\Bcal_0^{-1}(\Ecal)$ is then given by the trace of endomorphisms
	\begin{equation} \label{eqn:extension}
	\tr\Phi_{\alpha\beta}\,:\, \End_0(\Ecal) \,\lra\, \Omega_{X/T},\ \ \beta\,\longmapsto\,
	\tr(\beta\Ascr_{\alpha\beta}).
	\end{equation}
	since $\tr \beta\,=\,0$.
	
Finally, to complete the proof we must show that if $\Ecal\,=\,\Ecal_\phi$,
then $\tr\Phi_{\alpha\beta}$ vanishes on $\Kcal_\phi$. When $\Ecal\,=\,\Ecal_\phi$,
we may choose local holomorphic connections on $P$, so that $\Ascr_{\alpha\beta}$ 
is in the image of $\phi_\ast$. But $\Kcal_\phi$ consists precisely of endomorphisms that are orthogonal
to these under the trace pairing. Thus, the proposition follows from
the expression in \eqref{eqn:extension}.
\end{proof}

It is shown in \cite[Thm.\ B.2.6]{BBMP20} that the exact sequence for
$\Bcal_{0,X/T}^{-1}(\Ecal)$ is dual to the (traceless) Atiyah algebra sequence
for $\Ecal$ using the trace map. In the case of $\Ecal_\phi$, we have the
natural map $\widetilde\phi_\ast^t \,:\, \At_{X/T}(\Ecal_\phi)^\vee\,\longrightarrow
\, \At_{X/T}(P)^\vee$. Then, by Proposition \ref{prop:tricky1} we have the following:

\begin{corollary}
Let $P\,\longrightarrow\, X$ be a principal $G$-bundle, $\phi\,:\,G\,\longrightarrow\,\SL(V)$ a
holomorphic representation, and $\Ecal_\phi\,\longrightarrow\, X$ the associated vector bundle.
Then the degree $-1$ part of the Bloch-Esnault complex
$\Bcal^{-1}_{X/T}(\Ecal_\phi)$ is the pullback of the $-1$ part of the
Ginzburg complex in the commutative diagram
$$
\begin{tikzcd}
0 \arrow{r} &  \Omega_{X/T}\arrow{r}\arrow[d, equal] &
	\Bcal_{0,X/T}^{-1}(\Ecal_0) \arrow{r}\arrow[d, "\widetilde\phi_\ast^t"]
	&
	\End_0(\Ecal_\phi) \arrow{r}\arrow[d, "\phi_\ast^t"]& 0  \\
	0 \arrow{r} &  \Omega_{X/T}\arrow{r} &
	\Sfrak^{-1}_{X/T}(P) \arrow{r}
	&
	\ad(P) \arrow{r}& 0.
	\end{tikzcd}
	$$
\end{corollary}

Consider the map $\widetilde{\psi}\,:\, \At_{X/T}(\mathcal{E}_{\phi})\,\longrightarrow\,
\At_{X/T}(P)$ obtained in Proposition \ref{prop:wrongwaymap} along with its dual
$\widetilde{\psi}^*\,:\, \At_{X/T}(P)^{\vee}\,\longrightarrow\,
\At_{X/T}(\mathcal{E}_{\phi})$. We summarize the above discussions in the following commutative diagram:
\begin{equation}\label{eqn:impdia}
\begin{tikzcd}
0\arrow[r] &\Omega_{X/T} \arrow[r] \arrow[rddd,
"m_{\phi}\operatorname{Id}" description, bend right=5]&
\Sfrak^{-1}_{X/T}(P)\arrow[rd, "\cong" ]
\arrow[r]\arrow[rddd, "\widehat{\phi}" description, bend right=7] &
\ad(P)\arrow[rd, "\nu_\gfrak"]\arrow[rddd, "-\phi" description, bend right=7]\arrow[r]&0\\
&0 \arrow[r, crossing over]	& \Omega_{X/T} \arrow[from=lu, crossing over, equal] \arrow[r, crossing over]\arrow[d, "m_{\phi}\operatorname{Id}" description] &
\At_{X/T}(P)^{\vee} \arrow[r,crossing over] \arrow[d, "m_{\phi}\widetilde{\psi}^*"  description]& \ad(P)^{\vee}\arrow[d, "m_{\phi}\psi^t" description]\arrow[r] &0\\
&0 \arrow[r, crossing over]	& \Omega_{X/T} \arrow[d, equal] \arrow[r,
crossing over] & \At_{X/T}(\Ecal_\phi)^{\vee} \arrow[r, crossing
over]\arrow[d, "\cong"] &
\End_0(\Ecal_\phi)^{\vee}\arrow[d, "-\nu^{-1}_{\mathfrak{sl}(r)}"]\arrow[r] &0\\
&0\arrow[r]&	\Omega_{X/T} \arrow[r] & \mathcal{B}_{0,X/T}^{-1}(\Ecal_\phi)
\arrow[r]&  \End_0(\Ecal_\phi)\arrow[r]&0.\\
\end{tikzcd}
\end{equation}
In \eqref{eqn:impdia}, $\widehat{\phi}$ is just the map obtained by composition in the middle column of \eqref{eqn:impdia}. Observe that 
\begin{equation}
\nu_{\mathfrak{sl}(r)}^{-1}\circ m_{\phi} \psi^*\circ \nu_{\mathfrak{g}}\,=\,m_{\phi}(\nu_{\mathfrak{sl}(r)}^{-1}\circ \psi^* \circ
\nu_{\gfrak})\,=\,m_{\phi}(\frac{1}{m_{\phi}}\phi)\,=\,\phi.
\end{equation}
Composing the maps that appear in the above diagram we get the following.

\begin{proposition}\label{prop:veryimportant}
The map $\phi\,:\, \ad(P)\,\longrightarrow\, \End_0(\Ecal_\phi)$
extends to  $\widehat{\phi}\,:\, \Sfrak^{-1}_{X/T}(P)\,\longrightarrow\,
\Bcal_{0,X/T}^{-1}(\Ecal_\phi)$
which restricts to multiplication by the Dynkin index $m_{\phi}$ 
on $\Omega_{X/T}$. Taking push-forward by $R^1\pi$ yields the commutative diagram
\begin{equation}\label{eqn:veryimp}
\begin{tikzcd}
0 \arrow[r] & \mathcal{O}_{T} \arrow[r]\arrow[d,
"m_{\phi}\operatorname{Id}"] & R^1\pi_{*}(\Sfrak^{-1}_{X/T}(P))
\arrow[r]\arrow[d, "\widehat{ \phi}"] & R^1\pi_{*} (\ad(P))
\arrow[r] \arrow[d, " -\phi"]&0\\
0 \arrow[r] & \mathcal{O}_{T} \arrow[r] &
R^1\pi_{*}(\Bcal_{0,X/T}^{-1}(\Ecal_\phi)) 
\arrow[r] & R^1\pi_{*} (\End_0(\Ecal_\phi)) \arrow[r] &0.
\end{tikzcd}
\end{equation}
\end{proposition}

\subsubsection{The relative extension class} \label{sec:atiyah-bott}

In this section we consider the special case where $X\,\longrightarrow\, T$ is simply a product
$X\,=\,C\times T$. Let $P\,\longrightarrow\, X$ be a principal $G$-bundle such that the
restriction of it to $C\times \{t\}$ is regularly stable for every $t$. We wish to
compute the extension class of the top (and hence also the bottom) row of
\eqref{eqn:veryimp}. Since we assume the curve is fixed, for future
reference we call this the \emph{relative extension class}. 
To state the result, let $M_G^{rs}$ denote the moduli space of regularly
stable bundles on $C$. Since this a coarse moduli space, there is a morphism
$\varphi\,:\, T\,\longrightarrow\, M_G^{rs}$. By the deformation theory of 
principal $G$-bundles there is a homomorphism
$
\Tcal T\,\lra\, R^1\pi_\ast(\ad(P))
$.
Via the Dolbeault isomorphism, we have
\begin{equation*} \label{eqn:dolbeault}
H^1(T, \,(R^1\pi_\ast(\ad(P))^\vee)\,\lra\, H^1(T,\, (\Tcal T)^\vee)\,\simeq\,
H^{1,1}_{\dbar}(T)
\end{equation*}
and so the extension class of \eqref{eqn:veryimp} defines
a class in $H^{1,1}_{\dbar}(T)$. On the other hand, there is another natural
class 
$\varphi^\ast[\Theta]\,\in\, H^{1,1}_{\dbar}(T)$ due to Atiyah-Bott, which we
define below. The result is then:

\begin{theorem} \label{thm:atiyah-bott}
Let $X\,=\, C\times T$ as above. 
The image of the extension class of \eqref{eqn:veryimp}
in $H^{1,1}_{\dbar}(T)$ is exactly  the pullback  $\varphi^\ast[\Theta]$
of the Atiyah-Bott-Narasimhan-Goldman form. 
\end{theorem}

We first define $\Theta$. Under the assumptions, the direct image $R^1\pi_\ast(\ad(P))\,\longrightarrow\, T$ is
locally free and its fiber at $t\,\in\, T$ isomorphic to 
$H^1(C,\, \ad(P_t))$, where $P_t\,=\, P\vert_{C\times\{t\}}$. By Dolbeault isomorphism, this is isomorphic to
$H_{\dbar}^{0,1}(C,\,\ad(P_t))$.
 
Fix a maximal compact subgroup $K\,\subset\, G$. Let $\rho_K$ be the Cartan involution of $\mathfrak g$ that
fixes ${\rm Lie}(K)$ and acts on ${\rm Lie}(K)^\perp$ as multiplication by $-1$.
The Narasimhan-Seshadri-Ramanathan theorem (\cite{NarSesh}, \cite{Ramanathan:75}) asserts the existence of 
a $C^\infty$ reduction of structure group $P^t_K\,\subset\, P_t$ of $P_t$ to $K$,
satisfying the condition that the associated Chern connection on $P_t$ is flat.
The Chern connection defines harmonic representatives $\Hcal^{0,1}(C,\,\ad(P_t))$ of the Dolbeault
group $\Hcal_{\dbar}^{0,1}(C,\,\ad(P_t))$. 
The Cartan involution $\rho_K$ produces a conjugate linear involution
$\rho'$ of $\ad(P_t)\otimes T^*C\otimes {\mathbb C}$ that exchanges $\Omega^{1,0}_C(\ad(P_t))$ and
$\Omega^{0,1}_C(\ad(P_t))$ and preserves the harmonic forms; this $\rho'$ is also
called the Hodge $\ast$ operator. Then a hermitian inner product on 
$\Hcal^{0,1}(C,\,\ad(P_t))$ is given by
\begin{equation*}
\langle \alpha,\, \beta\rangle \,=\, \sqrt{-1}\int_C(\alpha\wedge\rho'(\beta))_\gfrak ,
\end{equation*}
where $\alpha$ and $\beta$ are harmonic representatives. The almost complex structure on $M_G^{rs}$
at $\varphi(t)$ is given by $\rho'$ on harmonic $1$-forms. The tangent space to the space of flat
$K$-connections at the point corresponding to the Narasimhan-Seshadri-Ramanathan connection on $P^t_K$
coincides with $H^1(C, \,\underline{\ad}(P^t_K))$, where $\underline{\ad}(P^t_K)$ is the local system.
We note that $H^1(C, \,\underline{\ad}(P^t_K))$ is identified with $H^{0,1}_{\dbar}(C,\,\ad(P))$ by the map
$\alpha \, \longmapsto\, \alpha-\sqrt{-1}\rho(\alpha)$.
The almost complex structure on $H^1(C, \,\underline{\ad}(P^t_K))$ is given by $\rho'$.
For $a,\, b\,\in\, H^1(C,\, \underline{\ad}(P^t_K))$, the Riemannian metric is given by 
$$
\langle a,\,b\rangle\,=\,-\int_C (a\wedge \ast b)_\gfrak\,=\, 2\text{Re}\,
\langle\alpha,\beta\rangle,
$$
where $\alpha,\, \beta\,\in\, H^{0,1}_{\dbar}(C,\ad(P))$ correspond to $a,\, b$ respectively.
Finally, the \emph{Atiyah-Bott-Narasimhan-Goldman}
symplectic form is given by 
$
\Theta(a,\,b)\,=\, 2\text{Im}\, \langle\alpha,\beta\rangle
$.
It is closed and of type $(1,\,1)$, and so defines a class in
$H^{1,1}_{\dbar}(T)$. 

\begin{proof}[Proof of Theorem \ref{thm:atiyah-bott}]
First, we note the following:
\begin{enumerate}
\item $\At(P)^\vee$ is the sheaf of invariant $1$-forms on
$P$, i.e.,\ $\varphi\,\in\, \Gamma(\pi^{-1}(U),\,T^\ast P)$ such that
$R_g^\ast\varphi\,=\,\varphi$ for all $g\,\in\, G$.

\item $\ad(P)^\vee$ is the sheaf of maps $P\,\longrightarrow\,
\gfrak^\ast$ that are equivariant with respect to the
co-adjoint action, i.e.,\ $f(pg)\,=\,f(p)\circ\Ad_g$. 

\item The map $j$ is defined explicitly by:
$j(\varphi)(p)(X)\,=\,\varphi(X^\sharp_p)$, for $X\,\in\, \gfrak$. We have
\begin{align*}
j(\varphi)(pg)(X)&=\varphi(pg)(X^\sharp_{pg})
=\varphi(pg)\left( (R_g)_\ast(\Ad_g X)^\sharp_p\right) \\
&=R_g^\ast\varphi((\Ad_g X)^\sharp_p) 
=j(\varphi)\circ\Ad_g(X)
\end{align*}
consequently $j(\varphi)$ satisfies the correct invariance. 
\end{enumerate}

A relative holomorphic connection $\omega$ on $P$ gives a holomorphic splitting
of \eqref{eqn:dualatiyahG}. Explicitly, if $\varphi$ is a local
section of $\ad(P)^\vee$, then define $\omega(\varphi)$ in
$\At(P)^\vee$, by $\omega(\varphi)(Y)\,=\,\varphi(\omega(Y))$. 
Notice that $$
j(\omega(\varphi))(p)(Y)\,=\,\omega(\varphi)(Y^\sharp_p)
\,=\,\varphi(\omega(Y^\sharp_p))\,=\,\varphi(Y),
$$
so this map is a splitting of the sequence in \eqref{eqn:dualatiyahG}. 
Let $M^{dR}_G$ denote the moduli space of holomorphic $G$-connections
on a fixed curve $C$. Let $M^{dR,rs}_G$ denote the open subset where the
underlying bundle $P$ is regularly stable. Hence there is the forgetful
map $p\,:\,M^{dR, rs}_G\,\longrightarrow\, M_G^{rs}$, and this
is a submersion. Therefore, we can find local holomorphic sections.
With this understood, let
$\{\alpha_j\}$  be a local holomorphic frame for the bundle
$R^1\pi_\ast(\ad(P)^\vee)$ over an open set 
$U\,\subset\, T$, and let $\{\alpha_j^\ast\}$ be the
dual frame. From  the previous paragraph, after shrinking $U$, we may 
find a holomorphic family of relatively flat connections
$\omega_U^{hol}$ for the restriction of $P$ to $C\times U$. Then a lift of the identity
endomorphism of  $R^1\pi_\ast(\ad(P)^\vee)$
to $R^1\pi_\ast(\At(P)^\vee)\otimes (R^1\pi_\ast(\ad(P)^\vee))^\vee$ is given by 
$$
s_U\,=\,\sum_j \omega_U^{hol}(\alpha_j)\otimes \alpha_j^\ast.
$$
For open subsets $U$, $V$, the difference $\sigma_{UV}\,=\,s_U-s_V$ is valued in 
$R^1\pi_\ast(\Omega_C)\otimes (R^1\pi_\ast(\ad(P)^\ast))^\ast$,
and the $1$-cocycle  $\{\sigma_{UV}\}$ represents the extension class. 

We now shift to the Dolbeault picture.
First, the identification $R^1\pi_\ast(\Omega_C)\,\simeq\, \CBbb$
is given by integration along the fiber $C$. 
Next, using the Killing form we identify $\ad P\,\simeq\, (\ad P)^\vee$.
Consider the map $\Tcal T\,\longrightarrow\,
R^1\pi_\ast(ad P)$; then as discussed above the extension class
defines via the Dolbeault isomorphism  a $\dbar$-closed  $(1,\,1)$-form on $T$. 
Let $P_K\,\subset \,P$ be the reduction of structure group given by the
Narasimhan-Seshadri-Ramanathan theorem. The Chern connection on $P_K$
extends to a connection $\omega_A$ on $P$ that restricts
to a flat connection on each $P_t$, although
$\omega_A$ does not vary holomorphically in $t\,\in\, T$. 
We can write $\omega^{hol}_U\,=\,\omega_A+B_U$, where $B_U$ is an
invariant $\gfrak$-valued relative $1$-form on $P$ that vanishes on
vertical vector fields. Let $b_U\,=\,
\sum_j B_U^{hol}(\alpha_j)\otimes \alpha_j^\ast$.
Then since $\omega_A$ is globally defined, we have $s_U-s_V\,=\,b_U-b_V$.
Hence, $\{b_U\}$ gives a $C^\infty$ trivialization of the 1-cocycle
$\{\sigma_{UV}\}$. By definition, the Dolbeault representative
of the extension class is therefore given by the global
$(0,\,1)$-form $\{\dbar b_U\}$. Now, since $\omega^{hol}_U$ is holomorphic, the
extension class is represented by 
$$
\Omega_A\,=\,- \sum_j\dbar\omega_A(\alpha_j)\otimes \alpha_j^\ast.
$$

We calculate this form at a given point $t\,\in\, U$. Choose local holomorphic
coordinates $s_1,\,\cdots,\, s_N$ centered at the point $[P_t]$ in $M_G^{rs}$
corresponding to $t$. We may so arrange that the
holomorphic sections $\alpha_j(s_1,\cdots, s_N)$ of $\Tcal M_G^{rs}$ in a neighborhood $[P_t]$
satisfy the condition $\alpha_j(0)\,=\,\partial/\partial s_j$.
Thus, we have also $\alpha_j^\ast(0)\,=\,ds_j$. Each $\alpha_j(s)$ defines a
Dolbeault class in $H^{0,1}(C,\, \ad \Pcal_s)$, for which we use the same
notation. Let $\rho_s(\alpha_j(s))$ be the hermitian conjugate with respect
to $\rho_s$. With this notation, the Chern connection is given by 
$$
\omega_A(s)\,=\,\omega_A(0)+\sum t_j(s)\alpha_j(s)+\sum \overline
t_j(s)\rho_s(\alpha_j(s)) \ \mod \ \Tcal^\vee M_G^{rs},
$$
where $t_j(0)\,=\,0$, $\partial t_j/\partial s_i(0)\,=\,\delta_{ij}$ and
$\partial t_j/\partial \overline s_i(0)\,=\,0$.
We therefore have 
\begin{align*}
\omega_A(\alpha_k)&\,=\, \omega_{A}(0)(\alpha_k)+\sum\overline
t_j(\rho(\alpha_j),\alpha_k) \\
\dbar    \omega_A(\alpha_k)(0)&\,=\, \sum (\rho(\alpha_j),\alpha_k)d\overline
s_j.
\end{align*}
Recalling that $H^1(C,\,K_C)\,\simeq\,\CBbb$ is gotten by integration over $C$
we have that $\Omega_A$ is the pullback of the form
$$
\sum \int_C (\rho(\alpha_j), \alpha_k) ds_k\wedge d\overline s_j
\,=\,i\sum \langle \alpha_k, \alpha_j\rangle ds_k\wedge d\overline s_j
	$$
and this is precisely the Atiyah-Bott-Narasimhan-Goldman symplectic form. 
\end{proof}

\section{Determinant of cohomology and Beilinson-Schechtman classes}\label{sec:determinant}

In this section we reinterpret Theorem
\ref{thm:atiyah-bott} in terms of the quasi-isomorphisms between the
Bloch-Esnault complex, the trace complex of Beilinson-Schechtman, and 
the Ginzburg complex. Finally we related all of these to the Atiyah class of the determinant of coholomology line bundle.

\subsection{Definitions} \label{sec:bs}
Recall the notation ${\mathcal{E}}'\,:=\,{\mathcal{E}}\otimes \Omega_{{X}/T}$. 
The sheaf $\mathcal{D}^{\leq 1}( \mathcal{E})$ of first order differential operators 
on $\Ecal$ can be identified with $\frac{\BOXE(2\D)}{\BOXE}$. There is an exact sequence of sheaves 
\begin{equation}
0 \,\lra\,\frac{\BOXE(\D)}{\BOXE} \,\lra\,\frac{\BOXE(2\D)}{\BOXE}
\,\lra\,\frac{\BOXE(2\D)}{\BOXE(\Delta)} \,\lra\, 0.
\end{equation}
Similarly,  we use another short exact sequence  from \cite{BS88}:
\begin{equation}\label{eqn:tr1}
0 \,\lra\,\frac{\BOXE}{\BOXE(-\D)} \,\lra\,\frac{\BOXE(2\D)}{\BOXE(-\D)}
\,\lra\,\frac{\BOXE(2\D)}{\BOXE} \,\lra\, 0.
\end{equation}
Since we have $\At_{X/T}(\Ecal)\,\subset\, \mathcal{D}^{\leq1}_{X/T}(\Ecal)$, pulling back \eqref{eqn:tr1} we get a
quasi-Lie algebra that fits into the following short exact sequence: 
\begin{equation}
0\, \lra\,\End(\Ecal)\otimes \Omega_{X/T} \,\lra\,{}^{tr}\widetilde{
	\mathcal{A}}_{X/T}(\Ecal)^{-1} \,\lra\,\At_{X/T}(\Ecal)\,\lra\, 0.
\end{equation}
Pushing it forward via the trace map $\tr\,:\, \End(\Ecal)\,\longrightarrow \,\mathcal{O}_X$ we get that 
\begin{equation}
0 \,\lra\,\Omega_{X/T} \,\lra\,{}^{tr} \mathcal{A}_{X/T}(\Ecal)^{-1}
\,\lra\,\At_{X/T}(\Ecal)\,\lra\, 0.
\end{equation}

Recall that 
$\widetilde{\mathcal{B}}_{X/T}^{-1}(\Ecal)\,:=\,\frac{\BOXE(\D)}{\BOXE(-\D)}$, and
define $\widetilde{\Bcal}_{0,X/T}^{-1}(\Ecal)$ to be the trace free version. Now
consider the pushforwards of $\widetilde{\mathcal{B}}_{X/T}^{-1}(\Ecal)$ and
$\widetilde{\Bcal}_{0,X/T}^{-1}(\Ecal)$ via $\operatorname{tr}$, and denote
them by ${\mathcal{B}}^{-1}_{X/T}(\Ecal)$ and $\Bcal_{0,X/T}^{-1}(\Ecal)$ respectively.
The natural inclusion map $$\End(\mathcal{E})\,\cong\, \frac{\BOXE(\Delta)}{\BOXE}
\,\hookrightarrow \,\frac{\BOXE(2\D)}{\BOXE}$$ fits the above objects into the following commutative diagram 
\begin{equation}\label{eqn:3rowdiagram}
\begin{tikzcd}
0 \arrow[r] & \End(\mathcal{E})\otimes \Omega_{{X}/T} \arrow[r] & {}^{tr}\widetilde{\mathcal{A}}_{X/T}(\mathcal{E})^{-1} \arrow[r] & \operatorname{At}_{{X}/T}(\Ecal) \arrow[r]& 0\\
0 \arrow[r]& \End(\Ecal)\otimes \Omega_{{X}/T} \arrow[r] \arrow[u, equal] & \widetilde{\mathcal{B}}_{X/T}^{-1}(\Ecal) \arrow[r]\arrow[u, hook]& \End(\Ecal) \arrow[r] \arrow[u, hook] & 0 \\
0 \arrow[r]& \End(\Ecal)\otimes \Omega_{{X}/T}
\arrow[r]\arrow[d, "\operatorname{tr}"] \arrow[u, equal] &
\widetilde{\Bcal}_{0,X/T}^{-1}(\Ecal) \arrow[r] \arrow[d,
"\operatorname{tr}_*"]\arrow[u, hook]& \End_0(\Ecal) \arrow[r] \arrow[u, hook] \arrow[d, equal]& 0 \\
0 \arrow[r] & \Omega_{{X}/T} \arrow[r] & \Bcal_{0,X/T}^{-1}(\Ecal)
\arrow[r]  & \End_0(\Ecal) \arrow[r] &0\\
0 \arrow[r] & \Omega_{{X}/T} \arrow[from=u, equal]\arrow[r]\arrow[from=uuu,  bend left=57 ,crossing over,"\operatorname{tr}" description ] & \mathcal{B}_{X/T}^{-1}(\Ecal) \arrow[r]  \arrow[from=u, hook] \arrow[from=uuu,  bend left=57 ,crossing over,"\operatorname{tr}_*" description ] & \End(\Ecal) \arrow[r] \arrow[from=u, hook]\arrow[from=uuu, bend left=57, crossing over, equal] &0.\\
\end{tikzcd} 
\end{equation}
\subsection{Relative set-up}
Now, consider the relative set-up. Therefore, we have a smooth scheme $T$ over $S$ with connected fibers, and
$\pi\,:\,X\,\longrightarrow\, T$ is a family of connected smooth curves of genus $g$. 
We have short the exact sequence:
\[
0 \,\longrightarrow\, \mathcal{T}_{X/T}\,\longrightarrow\, \mathcal{T}_{X/S}\,\longrightarrow\,
\pi^*\mathcal{T}_{T/S}\,\longrightarrow\, 0.
\]
Now as $\Ocal_T$ modules, $\pi^{-1}(\mathcal{T}_{T/S})\,\subset\, \pi^\ast(\mathcal{T}_{T/S})$. 
Define $\mathcal{T}_{X,\pi,S}\,:=\,d\pi^{-1}(\pi^{-1}(\mathcal{T}_{T/S}))$.  Then we have
\begin{equation}\label{eqn:fibersplit}
0\,\lra\, \mathcal{T}_{X/T} \,\lra\, \mathcal{T}_{X,\pi,S}\,\stackrel{d\pi}{\lra}\, \pi^{-1}
\mathcal{T}_{T/S}\,\lra\, 0.
\end{equation} 
Let $\operatorname{At}_{X,\pi,S}(\mathcal{E})$ denote the Atiyah algebra satisfying the fundamental exact sequence 
$$0 \lra \operatorname{End}(\mathcal{E})\rightarrow \operatorname{At}_{X,\pi,S}(\mathcal{E})\lra \mathcal{T}_{X,\pi,S}\lra 0$$ Observe that $\operatorname{At}_{X/T}(\mathcal{E})\subset \operatorname{At}_{X,\pi,S}(\mathcal{E})\subset \operatorname{At}_{X/S}(\mathcal{E})$.
\begin{definition}
	The {\em Beilinson-Schechtman trace complex}
	${}^{tr}\mathcal{A}_{X,\pi,S}(\mathcal{E})^{\bullet}$ is the $\Omega_{X/T}$-extension of the 
	dgla associated the quasi-Lie algebra
	${}^{tr}{\mathcal{A}}_{X/T}(\mathcal{E})^{-1}$, 
	i.e., ${}^{tr}\mathcal{A}_{X,\pi,S}(\mathcal{E})^{(0)}\,=\,\operatorname{At}_{X,\pi,S}(\mathcal{E})$, ${}^{tr}\mathcal{A}_{X,\pi,S}(\mathcal{E})^{-1}= {}^{tr}{\mathcal{A}}_{X/T}(\mathcal{E})^{-1}$ and
	${}^{tr}\mathcal{A}_{X,\pi,S}(\mathcal{E})^{(-2)}\,=\,\mathcal{O}_X$.
\end{definition}

Throughout the rest of this section we will have the assumption
that there is a
splitting of the short exact sequence in \eqref{eqn:fibersplit}.

This condition holds for example in the case of fiber products. 
We then use the splitting of \eqref{eqn:fibersplit} to pull-back the Atiyah algebra $\operatorname{At}_{X,\pi,S}(\mathcal{E})$ further via $\pi^{-1}\mathcal{T}_{T/S}$ to obtain $\mathcal{B}_{X,\pi,S}$:
$$
\begin{tikzcd}
0 \arrow{r} &  \operatorname{End}{\mathcal{E}} \arrow{r} \arrow[equal]{d} & \operatorname{At}_{X,\pi,S}(\mathcal{E})
\arrow{r} &
\mathcal{T}_{X,\pi,S} \arrow{r} & 0 \\
0 \arrow{r} &  \operatorname{End}{\mathcal{E}} \arrow{r} & \mathcal{B}_{X,\pi,S}^0(\mathcal{E}) \arrow{r}
\arrow[hook]{u} &
\pi^{-1}\mathcal{T}_{T/S} \arrow{r} \arrow[hook]{u}& 0.
\end{tikzcd}
$$
Let $\mathcal{B}_{0,X,\pi,S}^0(\mathcal{E})$ be the pushout of the exact sequence defining $\mathcal{B}^0_{X,\pi,S}(\mathcal{E})$ via
the quotient homomorphism $\operatorname{End}(\mathcal{E})\,\longrightarrow\, \operatorname{End}_0(\mathcal{E})$. 
\begin{definition}
The {\em Bloch-Esnault complex} $\mathcal{B}^{\bullet}_{X,\pi,S}(\mathcal{E})$ is the three term complex consisting of the locally free
sheaves  $\mathcal{B}^{0}_{X,\pi,S}(\mathcal{E})$, $\mathcal{B}^{-1}_{X/T}(\mathcal{E})$ and $\mathcal{B}^{-2}_{X,\pi,S}(\mathcal{E})\,\cong\, \mathcal{O}_X$
in degrees $0$, $-1$ and $-2$ respectively, and zero otherwise.
\end{definition}
 We will denote by $\mathcal{B}_{0,X/T,S}^{\bullet}$ the
traceless Bloch-Esnault complex as considered in \cite{ST} and \cite{BBMP20}.
\subsection{Determinant of cohomology}Let $\pi:X\rightarrow T$ be a family of smooth projective curves and let $\mathcal{E}$ be a vector bundle $X$. Then the object $R\pi_*\mathcal{E}$ in the bounded derived category of $T$ is represented by a two term complex $\mathcal{E}_0\rightarrow \mathcal{E}_1$. The determinant of cohomology $\mathcal{L}$ upto isomorphism is defined by 
$$\mathcal{L}:=\bigwedge^{top}\mathcal{E}_1\otimes \bigwedge^{top}\mathcal{E}_0^{\vee}$$
We refer the reader to
\cite[\S~6.3]{BMW1}
for more details on determinant of cohomology $\mathcal{L}$. We recall the following result of Beilinson-Schechtman, \cite{BS88}, that connects the Atiyah algebra of
$\operatorname{At}_{X/T}(\mathcal{L})$ and the trace complex of $\Ecal$ and also a result of Bloch-Esnault \cite{BE} connecting the
trace complex with the Bloch-Esnault complex. 

\begin{proposition}The Atiyah algebra $\operatorname{At}_{T/S}(\mathcal{L}^{-1})$ is isomorphic to the relative Atiyah algebra $R^0\pi_*({}^{tr}\mathcal{A}_{X,\pi,S}(\Ecal)^{\bullet})$.
If $\mathcal{B}_{X,\pi,S}^{\bullet}(\mathcal{E})$ is the relative Bloch-Esnault complex
associated to a family of curves $X\,\longrightarrow\, T$ parametrized by $T\rightarrow S$ such that the exact sequence in Equation 
\eqref{eqn:fibersplit} splits, then $\mathcal{B}^{\bullet}_{X,\pi,S}(\mathcal{E})$ is 
quasi-isomorphic to the trace complex ${}^{tr}\mathcal{A}_{X,\pi,S}(\mathcal{E})^{\bullet}$.
\end{proposition}

\subsection{Relative Ginzburg complex}

Let $\mathcal{P}$ be a principal ${G}$-bundle on $X\,\longrightarrow\, T\lra S$, and consider the Atiyah algebra
$\operatorname{At}_{X,\pi}(\Pcal)$ with the
fundamental exact sequence 
$$0 \,\lra\, \gfrak_{\mathcal{P}}\,\longrightarrow\, \operatorname{At}_{X,\pi,S}(\Pcal)\,\lra\, \mathcal{T}_{X,\pi,S}\,\lra\,0.$$
We then use the splitting of the exact sequence in \eqref{eqn:fibersplit} to pull-back further to define $\mathfrak{S}^0(\Pcal)$:
$$
\begin{tikzcd}
0 \arrow{r} &  \gfrak_{\mathcal{P}} \arrow{r} \arrow[equal]{d} & \operatorname{At}_{X,\pi,S}(\mathcal{P})
\arrow{r} &
\mathcal{T}_{X,\pi,S} \arrow{r} & 0 \\
0 \arrow{r} &  \gfrak_{\mathcal{P}} \arrow{r} & \mathfrak{S}^0_{X,\pi,S}(\mathcal{P}) \arrow{r}
\arrow[hook]{u} &
\pi^{-1}\mathcal{T}_{T/S} \arrow{r} \arrow[hook]{u}& 0.
\end{tikzcd}
$$
We consider the following three term complex $\mathfrak{S}^{\bullet}_X,\pi,S(\mathcal{P})$ which will be referred to as
the {\em relative Ginzburg complex}:
$$
\mathfrak{S}^{\bullet}_{X,\pi,S}(\Pcal)\,\,\,:=\,\,\begin{cases}
\mathcal{O}_X\ \ \mbox{if $i\,=\,-2$},\\
\mathfrak{S}^{-1}_{X/T}(\Pcal)\ \ \mbox{if $i\,=\,-1$},\\
\mathfrak{S}^{0}_{X,\pi,S}(\Pcal)\ \ \mbox{if $i\,=\,0$},\\
0,\ \ \mbox{otherwise}.
\end{cases}
$$
Let $\phi\,:\, {G}\,\longrightarrow\, \SL(V)$ be a holomorphic representation of dimension $r$, and let $\mathcal{E}_{\phi}$ be the
associated vector bundle. Then one recovers the Bloch-Esnault complex $\mathcal{B}_{0,X,\pi,S}^{\bullet}(\mathcal{E}_{\phi})$. 

We have the following commutative diagram in which the horizontal map is
the relative Ginzburg quasi-Lie algebra associated to a relative principal bundle $\mathcal{P}$ 
\begin{equation}
\begin{tikzcd}
&\mathcal{O}_{X} \arrow[equal]{r} \arrow{d} & \mathcal{O}_{X} \arrow[d] &\\
0\arrow[r]&\Omega_{X/T} \arrow{r} & \mathfrak{S}^{-1}_{X/T}(\mathcal{P})\arrow[r]\arrow[d]& \gfrak_{\mathcal{P}} \arrow[r]\arrow[d]& 0\\
&&\mathfrak{S}^{0}_{X,\pi,S}(\mathcal{P}) \arrow[equal]{r} & \mathfrak{S}^{0}_{X,\pi,S}(\mathcal{P}).
\end{tikzcd}
\end{equation}
The above diagram can be written as a short exact sequence of complexes 
\begin{eqnarray}
0 \,\longrightarrow \,\Omega^{\bullet}_{X/T}[2]\,\longrightarrow\, \mathfrak{S}^{\bullet}_{X,\pi,S}(\mathcal{P})\,\longrightarrow\,
\mathfrak{R}^{\bullet}(\mathcal{P})\,:=\,(\gfrak_{\mathcal{P}}\,\longrightarrow\, \mathfrak{S}^0_{X,\pi,S}(\mathcal{P}))\,\longrightarrow\, 0.
\end{eqnarray}
The complex $\mathfrak{R}^{\bullet}(\mathcal{P})$ is quasi-isomorphic to
$\pi^{-1}\mathcal{T}_{T/S}$. Hence, there is a natural map of complexes $
\pi^{-1}\mathcal{T}_{T/S}\,\cong\, \mathfrak{R}^{\bullet}(\mathcal{P})\,\longrightarrow\, \mathfrak{g}_{\mathcal{P}}[1]$. Taking
pushforward with $\pi$, we get a map $\mathcal{T}_{T/S}\,\cong\, \pi_*(\pi^{-1}\mathcal{T}_{T/S})
\,\longrightarrow\, R^1\pi_{n*}\mathfrak{g}_{\mathcal{P}}.$ We wish to compute $R^0\pi_\ast \mathfrak{S}_{X,\pi,S}^\bullet(\mathcal{P})$.  

To compute $R^0\pi_\ast \mathfrak{S}_{X,\pi,S}^\bullet(\mathcal{P})$, we need to compute the zero-th hypercohomology. 
Choose fine resolutions (local on $\Mcal$), 
$$
\begin{tikzcd}
\Gfrak^{-2}(\mathcal{P})\arrow["f"]{r} \arrow{d}
&\Gfrak^{-1}(\mathcal{P})\arrow["g"]{r}\arrow{d}& \Gfrak^0(\mathcal{P}) \arrow{d}\\
\Lcal_0(\mathcal{P})\arrow["f"]{r}\arrow["d"]{d} &\Mcal_0(\mathcal{P})\arrow["g"]{r}\arrow["d"]{d}& \Ncal_0(\mathcal{P}) \arrow["d"]{d}\\
\Lcal_1(\mathcal{P})\arrow["f"]{r} \arrow["d"]{d}&\Mcal_1(\mathcal{P})\arrow["g"]{r}\arrow["d"]{d}& \Ncal_1(\mathcal{P}) \arrow["d"]{d}\\
\vdots &\vdots&\vdots
\end{tikzcd}
$$
and compute the zero-th cohomology of the total complex
$
\begin{tikzcd}
\Ccal_{-1}\arrow["D_{-1}"]{r} &\Ccal_0\arrow["D_0"]{r} &\Ccal_1, \mbox{where}
\end{tikzcd}
$

$$
\Ccal_{-1}\,=\,\Lcal_1\oplus \Mcal_0\ , \ \Ccal_0\,=\,\Lcal_2\oplus\Mcal_1\oplus
\Ncal_0\ ,\ \Ccal_1\,=\,\Lcal_3\oplus\Mcal_2\oplus\Ncal_1,  \ \mbox{and the differentials are}
$$
\begin{align*}
&D_{-1}\,:\, \Ccal_{-1}\,\longrightarrow\, \Ccal_0,\,\  (\ell_1,\, m_0)\,\longmapsto\, (d\ell_1,\, dm_0-f\ell_1,\,
gm_0) \\
&D_0\,:\, \Ccal_{0}\,\longrightarrow\, \Ccal_1,\,\  (\ell_2,\, m_1,\, n_0)\,\longmapsto\, (d\ell_2,\, dm_1+f\ell_2,\,
dn_0-gm_1).
\end{align*}

We note three facts in the next lemma the proofs of which are immediate in view of the assumptions.

\begin{lemma}\label{lem:obviousbutimp}
Let $\pi\,:\, X \,\longrightarrow \,T$ be as above. Then the following hold:
\begin{enumerate}
\item $R^2\pi_\ast\Ocal_X\,=\,\{0\}$ (since the relative dimension is $1$).

\item $R^0\pi_\ast\gfrak_\mathcal{P}\,=\,\{0\}$ (assuming the stable locus is non-empty). 

\item Since the fibers are connected, assume that the natural map 
\begin{equation}\label{eqn:natural}
\pi_\ast(\pi^{-1}\mathcal{T}_{T/S})\,\cong\, \mathcal{T}_{T/S}\to R^1\pi_\ast\gfrak_P
\end{equation}
is an isomorphism. Then the map $R^1\pi_\ast(\mathfrak{S}_{X/T}^{-1}(\mathcal{P}))\,
\longrightarrow\, R^1\pi_\ast(\mathfrak{S}_{X,\pi,S}^0(\mathcal{P}))$
is zero as it factors through $R^1\pi_\ast\gfrak_\mathcal{P}$. 
\end{enumerate}
\end{lemma}

We prove the next proposition under the assumption of Statement (iii) in Lemma \ref{lem:obviousbutimp}.

\begin{proposition}\label{prop:old1}
	There is an isomorphism
	$R^0\pi_{*}\mathfrak{S}_{X,\pi,S}^{\bullet}(\mathcal{P})\simeq
	R^1\pi_{*}\mathfrak{S}_{X/T}^{-1}(\mathcal{P})$ that is the
	identity on $\mathcal{O}_{T}$ and $\mathcal{T}_{T/S}$. 
\end{proposition}
\begin{proof}
	Let $(\ell_2,m_1,n_0)\in \ker D_0$. By Lemma \ref{lem:obviousbutimp}(i) we may take $\ell_2=0$.
	Hence, $(0,m_1,n_0)\mapsto (0, dm_1, dn_0-gm_1)=(0,0,0).$
	So $m_1$ defines a class in $R^1\pi_\ast(\mathfrak{S}_{X/T}^{-1}(\mathcal{P}))$. The second
	condition says that $gm_1$ defines the zero class in
	$R^1\pi_\ast(\pi^{-1}\mathcal{T}_{T/S})$. But by Part (iii) of Lemma \ref{lem:obviousbutimp} this is automatic. 
	Now $n_0$ is such that $dn_0=gm_1$. 
	By (ii)  and (iii) of Lemma \ref{lem:obviousbutimp},  we have $R^0\pi_\ast\mathfrak{S}_{X,\pi,S}^0(\mathcal{P})=\{0\}$. This means that
	$n_0$ is uniquely determined. Hence, the hypercohomology gives
	$R^1\pi_\ast\mathfrak{S}_{X/T}^{-1}(\mathcal{P})$. 
\end{proof}

Pick a representation $\phi: \gfrak \rightarrow \mathfrak{sl}_r$, and
consider the adjoint 
$\ad : \mathfrak{sl}_r \rightarrow \mathfrak{sl}(\mathfrak{sl}_r)$. 
This gives a map of the corresponding simply connected groups, and now the 
corresponding associated construction first via $\phi$ gives a vector 
bundle $\mathcal{E}_{\phi}$ and then taking the adjoint we have the bundle 
$\operatorname{End}(\mathcal{E}_{\phi})$. 
Assuming the condition stated in (iii) of Lemma \ref{lem:obviousbutimp}, 
we have the following proposition,  which is a generalization of  
results in \cite[Proposition 5.0.2]{BBMP20} and \cite{ST}:

\begin{proposition}\label{prop:today2}
There is a natural isomorphism between
$R^0\pi_{*}\mathcal{B}^{\bullet}_{X,\pi,S}(End_0(\mathcal{E}_{\phi}))$
	of the Bloch-Esnault complex  and 
	$R^{0}\pi_{*}\mathfrak{S}_{X,\pi,S}^{\bullet}(\mathcal{P})$ of the
	relative Ginzburg complex, and the isomorphism fits in the following diagram: 
	\begin{equation}
	\begin{tikzcd}
	0 \arrow[r] &
	R^0\pi_*\Omega^{\bullet}_{X/T}[2]\cong\mathcal{O}_T\arrow[r]\arrow
	["m_\phi\cdot 2r"]{d} &R^{0}\pi_{*}\mathfrak{S}^{\bullet}_{X,\pi,S}(\mathcal{P})\arrow[r]\arrow["\cong"]{d} & \mathcal{T}_{T/S} \arrow[r]\arrow["-\id"]{d} &0\\
	0 \arrow[r] & R^0\pi_*\Omega^{\bullet}_{X/T}[2]\cong\mathcal{O}_T\arrow[r] &R^{0}\pi_{*}\mathcal{B}^{\bullet}_{X,\pi,S}(\operatorname{End}_0(\mathcal{E}_{\phi}))\arrow{r} & \mathcal{T}_{T/S} \arrow[r] &0.
	\end{tikzcd}
	\end{equation}
\end{proposition}

\begin{proof}
The composition of maps
$\ad \circ \phi\,:\, \mathfrak{g}\,\longrightarrow\, \mathfrak{sl}(\mathfrak{sl}_r)$ 
gives a  map $\At_{X/T}(\mathcal{P})\,\longrightarrow\,
\At_{X/T}(\operatorname{End}_0(\mathcal{E}_{\phi}))$ (see \eqref{eqn:atiyah} and \cite{Atiyahconn}).
Moreover, the representation also gives the map
$\mathfrak{S}_{X/T}^{-1}(\mathcal{P})\,\longrightarrow\,
\mathcal{B}^{-1}_{0,X/T}(\operatorname{End}(\mathcal{E}_{\phi}))$ (see Proposition 
\ref{prop:veryimportant}). These two together give a map between the complexes 
	\begin{equation}
	\begin{tikzcd}
	0 \arrow{r}& \Omega^{\bullet}_{X/T}[2] \arrow{r} \arrow["m_{\phi}\cdot 2r"]{d}&\mathfrak{S}_{X,\pi,S}^{\bullet}(\mathcal{P}) \arrow{d}\arrow{r} &  \pi^{-1}\mathcal{T}_{T/S} \arrow{r}\arrow["-\id"]{d}& 0\\
	0 \arrow{r}& \Omega^{\bullet}_{X/T}[2] \arrow{r}  &\mathcal{B}^{\bullet}_{X,\pi,S}(\operatorname{End}(\mathcal{E}_{\phi})) \arrow{r} &  \pi^{-1}\mathcal{T}_{T/S} \arrow{r}& 0.\\
	\end{tikzcd}
	\end{equation}
Taking pushforward $R^1\pi_*$ of the above,  we obtain the  desired result. 
\end{proof}

Now assume that $R^1\pi_*\gfrak_{\mathcal{P}}$ is isomorphic to $\mathcal{T}_{T/S}$ under the natural map
in \eqref{eqn:natural}. Consequently, combining Propositions \ref{prop:old1} and \ref{prop:today2}, we have the following diagram 
\begin{equation}
\begin{tikzcd}
0\arrow{r}& R^1\Omega_{X/T}\cong \mathcal{O}_T \arrow[equal]{d}\arrow{r} & R^1\pi_*\mathfrak{S}_{X/T}^{-1}(\mathcal{P}) \arrow{r} \arrow["\cong"]{d}& \mathcal{T}_{T/S}\arrow{r}\arrow[equal]{d}& 0 \\
0 \arrow[r] & R^0\pi_*\Omega^{\bullet}_{X/T}[2]\cong\mathcal{O}_T\arrow[r]\arrow ["m_\phi.2r"]{d} &R^{0}\pi_{*}\mathfrak{S}_{X,\pi,S}^{\bullet}(\mathcal{P})\arrow[r]\arrow["\cong"]{d} & \mathcal{T}_{T/S} \arrow[r]\arrow["-\id"]{d} &0\\0 \arrow[r] & R^0\pi_*\Omega^{\bullet}_{X/T}[2]\cong\mathcal{O}_T\arrow[r] &R^{0}\pi_{*}\mathcal{B}_{X,\pi,S}^{\bullet}(\operatorname{End}_0(\mathcal{E}_{\phi}))\arrow{r} & \mathcal{T}_{T/S} \arrow[r]\arrow[equal]{d} &0\\
0 \arrow[r] & \mathcal{O}_T\arrow[equal]{u} \arrow{r} &
R^0\pi_*({}^{tr} \mathcal{A}_{X,\pi,S}(\operatorname{End}_0(\mathcal{E}_{\phi}))^{\bullet})
\arrow["\cong"]{u}\arrow{r} & 
\mathcal{T}_{T/S}\arrow [r]\arrow[equal]{d}&0\\
0 \arrow[r] & \mathcal{O}_T\arrow[equal]{u} \arrow{r} & \operatorname{At}_{T/S}(\mathcal{L}^{-\otimes 2r}_{\phi})\arrow["\cong"]{u}\arrow{r} & \mathcal{T}_{T/S}\arrow [r]&0.
\end{tikzcd}
\end{equation}
The isomorphism between the third  and fourth rows is due to Bloch-Esnault
\cite{BE}, and the isomorphism between the fourth and fifth rows is due to Beilinson-Schechtman \cite{BS88}. 
Here, $\mathcal{L}_{\phi}$ is the determinant of cohomology associated to the family $\mathcal{E}_{\phi}$.

We have the following theorem under the assumption that $\mathcal{T}_{T/S}\,\cong\,
R^1\pi_*\gfrak_{\mathcal{P}}$ for the map in \eqref{eqn:natural}.

\begin{theorem}\label{thm:atiyahG}
The relative Atiyah sequence for 
$\frac{1}{m_{\phi}}\At_{T/S}(\mathcal{L}_{\phi})$ is isomorphic to 
\begin{equation}
0 \,\lra\,\mathcal{O}_{M^{rs}_{G}}\,\lra\,
R^1\pi_{n, \ast}(\Sfrak^{-1}_{X/T}(\mathcal{P}))\,\lra\, 
\mathcal{T}_{T/S}\,\lra\, 0\ .
\end{equation}
\end{theorem}

This justifies the computation of the relative extension class of the short exact
sequence via Dolbeault methods in Theorem \ref{thm:atiyah-bott}.

\subsection{Associated bundles and pullback} \label{sec:assoc-det}

In this section we discuss the relation between the relative Ginzburg 
complex for the moduli space $M_{G}$ and the pull-back 
of Bloch-Esnault complex of $M_{\SL_r}$ associated to a 
representation $\phi\,:\, G\,\hookrightarrow\, \SL_r$.
As previously, let $M^{rs}_{G}$  denote the moduli space parametrizing the
regularly stable principal $G$-bundles,  and  denote by $M^s_{\SL_r}$ the moduli space of
stable rank $r$ vector bundles of trivial determinant.
Let $\phi\,:\,G\,\hookrightarrow\,\SL_r$ be such that $G$ is not contained in 
any proper parabolic subgroup of $\SL_r$. With
this condition, we know from \cite{Bisstable} that a stable
principal ${G}$-bundle $P$ produces a stable $\SL_r$ bundle.

Consider the map $\phi\,:\,M^{rs}_{G}\,\longrightarrow\, M^s_{\SL_r}$
taking $[P]$ to $[\Ecal_\phi]$, where $\Ecal_\phi\,=\,P\times_\phi\CBbb^r$ is
the associated vector bundle.
Let $\Escr\to M^{s}_{\SL_r}\times_S \Ccal$ and $\Pscr\to M^{rs}_G\times_S \Ccal$  be the
universal bundles (since they exist on a covering by \'etale open subsets,
we can treat as if they exist). Then the associated vector bundle 
$\Escr_\phi\,=\,\Pscr\times_\phi\CBbb^r$ is the pull-back of $\Escr$ via $\phi$. 
We have the following diagram of maps:
\begin{equation}
\begin{tikzcd}
\Escr_\phi \arrow[rr]\arrow[d]&& \Escr \arrow[d]&\\
\mathfrak{X}_{G}:=M^{rs}_{G}\times_{S}\Ccal \arrow[rr,"f"]
\arrow[rd, "\pi_{n,G}"]\arrow[dd,"\pi_{w,{G}}"]&& \mathfrak{X}_{\SL_r}:=
M^s_{\SL_r}\times_S \Ccal \arrow[rd, "\pi_{n,\SL_r}"]\arrow[dd, "\pi_{w,\SL_r}" near start]&\\
&M^{rs}_{G}\arrow[rr,"\phi", crossing over, near start]&& M^s_{\SL_r}\\
\Ccal \arrow[rr, equal]\arrow[rd, "\pi_{s,G}"] && \Ccal \arrow[rd, "\pi_{s,\SL_r}"] & \\
&S \arrow[rr, equal]
\arrow[from=uu, crossing over,"\pi_{e,G}" near end]&& S\arrow[from=uu, crossing over, "\pi_{e,\SL_r}"].
\end{tikzcd}
\end{equation}
Recall that we have the following commutative diagram on $M^s_{\SL_r}$ that connects the Atiyah algebra
of the determinant bundle with the Atiyah algebra defined by Bloch-Esnault \cite{BE}
\begin{equation}
\begin{tikzcd}0 \arrow[r] &   \Omega_{\mathfrak{X}_{\SL_r}/M^s_{\SL_r}}
\arrow[r] \arrow[d, equal]&
\operatorname{At}_{\mathfrak{X}_{\SL_r}/M^s_{\SL_r}}(\Escr)^{\vee} \arrow[r] & \End_0(\Escr)^{\vee} \arrow[r] &0\\
0 \arrow[r] &   \Omega_{\mathfrak{X}_{\SL_r}/M^s_{\SL_r}} \arrow[r] &
\Bcal_{_{\mathfrak{X}_{\SL_r}/M^s_{\SL_r}}}^{-1}(\Escr) \arrow[r] \arrow[u, "\cong"]& \End_0(\Escr) \arrow[r]\arrow[u,  "-\nu_{\mathfrak{sl}(r)}", "\cong"']&0.\\
\end{tikzcd}
\end{equation}

Taking pushforward and combining the results of Beilinson-Schechtman \cite{BS88},
Bloch-Esnault, \cite{BE}, and Baier-Bolognesi-Martens-Pauly, \cite{BBMP20}, we get the following
commutative diagram of maps
\begin{equation*} 
\begin{tikzcd}
0 \arrow[r] &  R^1\pi_{n,\SL_r*}  \Omega_{\mathfrak{X}_{\SL_r}/M^s_{\SL_r}}
\arrow[r] \arrow[d, equal]&
R^1\pi_{n,\SL_r*}(\operatorname{At}_{\mathfrak{X}_{\SL_r}/M^s_{\SL_r}}(\Escr)^{\vee})
\arrow[r] & R^1\pi_{n,\SL_r*}(\End_0(\Escr)^{\vee}) \arrow[r] &0\\
0 \arrow[r] &  R^1\pi_{n,\SL_r*}  \Omega_{\mathfrak{X}_{\SL_r}/M^s_{\SL_r}}
\arrow[r] & R^1\pi_{n,\SL_r*}(\Bcal_{0,{\mathfrak{X}_{\SL_r}/M^s_{\SL_r}}}^{-1}(\Escr))
\arrow[r] \arrow[u, "\cong"]& R^1\pi_{n,\SL_r*}(\End_0(\Escr)) \arrow[r]\arrow[u,  "-\nu_{\mathfrak{sl}(r)}", "\cong"']&0\\
0 \arrow[r] &
\mathcal{O}_{M^s_{\SL_r}}\arrow[r]\arrow[u,"\cong"]&\operatorname{At}_{M^s_{\SL_r}/S}(\mathcal{L}^{-1})\arrow[r]\arrow[u,"\cong"]
& \mathcal{T}_{M^s_{\SL_r}/S}\arrow[r]\arrow[u,equal] &0
\end{tikzcd}
\end{equation*}
where $\mathcal{L}\,=\,\det (R\pi_{n*}\Escr)$ is the
determinant of cohomology of the family $\Escr_{\phi}$.
Pulling back the exact sequence at the bottom of the diagram by the
map  $\phi\,:\, M^{rs}_{G}\,\longrightarrow\, M^s_{\SL_r}$, we get the exact sequence
\begin{equation*}
0\lra\mathcal{O}_{M^{rs}_G}
\lra\operatorname{At}_{M^{rs}_G/S}(\phi^*(\mathcal{L}))
\lra\mathcal{T}_{M^{rs}_{G}/S} \lra0.
\end{equation*}
We wish to connect the following two
exact sequences
\begin{equation}\label{eqn:labelphi}
0 \lra \phi^*(R^1\pi_{n,\SL_r*}
\Omega_{\mathfrak{X}_{\SL_r}/M_{\SL_r}}) \lra
\phi^{*}(R^1\pi_{n,\SL_r*}(\Bcal_{0,{\mathfrak{X}_{\SL_r}/M^s_{\SL_r}}}^{-1}(\Escr)))
\lra\phi^* R^1\pi_{n,\SL_r*}(\End_0(\Escr))\lra 0,
\end{equation}
$$
0 \,\lra\,\mathcal{O}_{M^{rs}_G}\,\cong\, R^1\pi_{n*}
\Omega_{\mathfrak{X}_G/M^{rs}_G} \,\lra\,
R^1\pi_{n*}(\Sfrak^{-1}_{\mathfrak{X}_{\SL_r}/M^s_{\SL_r}}(\Pscr)) \,\lra\,\mathcal{T}_{M^{rs}_G/S} \,\lra\, 0.$$

By Proposition \ref{prop:veryimportant},
it is enough to relate \eqref{eqn:labelphi} with the following (see \eqref{eqn:veryimp}):
\begin{equation*}
0 \,\lra\, \mathcal{O}_{M^{rs}_G} \,\lra\,
R^1\pi_{n,G\, \ast}(\Bcal_{0,{\mathfrak{X}_{\SL_r}/M^s_{\SL_r}}}^{-1}(\Escr_\phi))
\,\lra\, R^1\pi_{n*}(\End_0(\Escr_\phi))\,\lra\, 0.
\end{equation*}

\begin{proposition}\label{conj:firstone}
There is an isomorphism 
$$R^1\pi_{n,G\, \ast}(\mathcal{B}_{0,\mathfrak{X}_{\SL_r}/M^s_{\SL_r}}^{-1}(\Escr_\phi))\,\,\simeq\,\,
	\phi^*(R^1\pi_{n, \SL_r*}(\mathcal{B}_{0,\mathfrak{X}_{\SL_r}/M^s_{\SL_r}}^{-1}(\Escr)))\ .$$
\end{proposition}

\begin{proof}
	Consider the following digram 
	\begin{center}
		\begin{tikzcd}
			\Escr_\phi\simeq f^*(\Escr)
			\arrow[r] \arrow[d]& \Escr\arrow[d]\\ 
			M^{rs}_{G}\times_{S}C=:\mathfrak{X}_{G}
			\arrow[r, "f"]\arrow[d,"\pi_{n,G}"]&
			\mathfrak{X}_{\operatorname{\SL_r}}:=M^s_{\SL_r}\times_{S}C
			\arrow[d,"\pi_{n,\SL_r}"]	\\
			M^{rs}_{G} \arrow[r, "\phi"] & M^s_{\SL_r}.
		\end{tikzcd}
	\end{center}
	
	This induces a map  of the following exact sequences of Atiyah
	algebras as in \cite{Atiyahconn}:
	\begin{equation*}
	\begin{tikzcd}\label{eqn:res1}
	0 \arrow[r] & \End_0(\Escr_\phi) \arrow[r] \arrow[d,"\cong"] &
	\operatorname{At}_{\mathfrak{X}_{G}/M^{rs}_{G}}(\Escr_\phi)
	\arrow[r] \arrow[d, "\cong"] &
	\mathcal{T}_{\mathfrak{X}_{G}/M^{rs}_{G}}
	\arrow[r] \arrow[d, "\cong"]& 0\\
	0 \arrow[r] & f^* \End_0(\Escr)
	\arrow[r] & f^*\operatorname{At}_{\mathfrak{X}_{\SL_r}/M^s_{\SL_r}}(\Escr)
	\arrow[r] &f^*\mathcal{T}_{\mathfrak{X}_{\SL_r}/M^s_{\SL_r}} \arrow[r] &0.
	\end{tikzcd}
	\end{equation*}
	
	Dualizing, we obtain	
	\begin{equation*}
	\begin{tikzcd}
	0 \arrow[r] & f^* \Omega_{\mathfrak{X}_{\SL_r}/M^s_{\SL_r}} \arrow[r]&
	f^*(\operatorname{At}_{\mathfrak{X}_{\SL_r}/M^s_{\SL_r}}(\Escr)^{\vee}) \arrow[r] &
	f^*((\End_0(\Escr))^{\vee})\arrow[r] &0\\
	0 \arrow[r] &  \Omega_{\mathfrak{X}_{G}/M^{rs}_{G}} \arrow[r]\arrow[from=u] \arrow[d, equal] &\operatorname{At}_{\mathfrak{X}_{G}/M^{rs}_{G}}
	(\Escr_\phi)^{\vee} \arrow[r]\arrow[from=u, "\cong"]
	\arrow[d, "q_{\Escr_\phi}^{-1}"',"\cong"]&
	(\End_0(\Escr_\phi))^{\vee}\arrow[r]\arrow[from=u] \arrow[d, "\cong","-\nu_{\mathfrak{sl}(r)}^{-1}"'] &0\\
	0 \arrow[r]
	&{\Omega}_{\mathfrak{X}_{G}/M^{rs}_{G}}
	\arrow[r] & \Bcal_{0,{\mathfrak{X}_{G}/M^{rs}_{G}}}^{-1}(\Escr_\phi) \arrow[r] & \End_0(\Escr_\phi))
	\arrow[r]& 0\\
	0 \arrow[r] & f^*{\Omega}_{\mathfrak{X}_{\SL_r}/M^s_{\SL_r}} \arrow[r]
	\arrow[uuu, bend right=65, crossing over]&
	f^*(\Bcal_{0,{\mathfrak{X}_{\SL_r}/M^s_{\SL_r}}}^{-1}(\Escr))\arrow[r] \arrow[uuu,
	"q_{\Escr}" description, bend right=65,, crossing
	over]& f^* \End_0(\Escr)\arrow[r]\arrow[uuu, bend right=65, crossing over,"-\nu_{\mathfrak{sl}(r)}" description] &0.
	\end{tikzcd}
	\end{equation*}
	Hence,  by composing we get the following diagram
	\begin{equation*}
	\begin{tikzcd}\label{eqn:res2}
	0 \arrow[r]
	&{\Omega}_{\mathfrak{X}_{G}/M^{rs}_{G}}
	\arrow[r] \arrow[from=d,"\cong"] &
	\Bcal_{0,\mathfrak{X}_{G}/M^{rs}_{G}}^{-1}(\Escr_\phi) \arrow[r] \arrow[from=d,
	"\cong"] & \End_0(\Escr_\phi) \arrow[r] \arrow[from=d, "\cong"]& 0\\
	0 \arrow[r] & f^*{\Omega}_{\mathfrak{X}_{\SL_r}/M^s_{\SL_r}} \arrow[r] &
	f^*(\Bcal_{0,{\mathfrak{X}_{\SL_r}/M^s_{\SL_r}}}^{-1}(\Escr))\arrow[r] & f^* \End_0(\Escr)\arrow[r] &0.
	\end{tikzcd}
	\end{equation*}
	
Applying $R^1\pi_{n,G\, \ast}$, we conclude that the following  extensions are isomorphic:
	\begin{equation*}
	\begin{tikzcd}
	0 \arrow[r] & R^1\pi_{n,G\, \ast}
	\Omega_{\mathfrak{X}_{G}/M^{rs}_{G}}
	\arrow[r] &
	R^1\pi_{n,G\, \ast}(\Bcal_{0,\mathfrak{X}_{G}/M^{rs}_{G}}^{-1}(\Escr_\phi))
	\arrow[r] & R^1\pi_{n,G\, \ast}
	(\End_0(\Escr_\phi)) \arrow[r]  &0\\
	0 \arrow[r]& R^1\pi_{n,G\, \ast} f^*
	\Omega_{\mathfrak{X}_{\SL_r}/M^s_{\SL_r}}
	\arrow[u,"\cong"]\arrow[r] &
	R^1\pi_{n,G\, \ast}(f^*(\Bcal_{0,\mathfrak{X}_{\SL_r}/M^s_{\SL_r}}^{-1}(\Escr)))
	\arrow[r] \arrow[u]& R^1\pi_{n,G\, \ast}f^*(\End_0(\Escr))
	\arrow[r]  \arrow[u,"\cong"]& 0.
	\end{tikzcd}
	\end{equation*}

Thus to finish the proof of the proposition we need to show that 
\begin{equation*}
R^1\pi_{n, {G}*} (f^*(\Bcal_{0,\mathfrak{X}_{\SL_r}/M^s_{\SL_r}}^{-1}(\Escr)))
\,\cong\, \phi^*(R^1\pi_{n, \SL_r*}(\Bcal_{0,\mathfrak{X}_{G}/M^{rs}_{G}}^{-1}(\Escr))).
	\end{equation*}
	But since the fibers of the morphism $\pi_{n, \SL_r}$ are smooth
	projective  curves, we conclude
	that $R^2\pi_{n, \SL_r*}\mathcal{F}$ is zero for any coherent
	sheaf $\mathcal{F}$ which is flat over $M_{\SL_r}$. Moreover,
	$R^1\pi_{n, \SL_r*}(\Bcal_0^{-1}(\Escr))$
	is locally free. Hence by base change of
cohomologies for flat morphisms, we get the required isomorphism. This completes the proof. 
\end{proof}

Consider the following diagram where $m_{\phi}$ be the Dynkin index of the
embedding $\mathfrak{g}\,\longrightarrow\, \mathfrak{sl}(V)$:
\begin{equation*}
\begin{tikzcd}[row sep=2em, column sep=1em]
0 \arrow[r] & \mathcal{O}_{M_G} \arrow[d,equal] \arrow[r] & 
\phi^*\operatorname{At}_{M_{\SL_r}^{s}/S}(\mathcal{L}^{-1}) \arrow[r]\arrow[d,
"\cong"']& \phi^*\mathcal{T}_{M^s_{\SL_r}/S} \arrow[r]\arrow[d, "KS_{M_{\SL_r}^{s}/S}"] &0\\
0\arrow[r]&\phi^*\left(R^1\pi_{n,\SL_r\, \ast}
\Omega_{\mathfrak{X}_{\SL_r}/M^s_{\SL_r}}\right)
\arrow[r]& \phi^{*}\left(R^1\pi_{n,\SL_r*}(\Bcal_{0,\mathfrak{X}_{\SL_r}/M^s_{\SL_r}}^{-1}(\Escr)\right)
\arrow[r]& \phi^*\left(R^1\pi_{n,\SL_r*}(\End_0(\Escr)\right)\arrow[r]\arrow[d, equal] &0 \\
0\arrow[r]&R^1\pi_{n,G\,\ast}
\Omega_{\mathfrak{X}_{G}/M^{rs}_{G}}
\arrow[r]\arrow[u,equal]&
R^1\pi_{n,G\, \ast}(\Bcal_{0,{\mathfrak{X}_{G}/M^{rs}_{G}}}^{-1}(\Escr_\phi))
\arrow[r]\arrow[u, "\cong"]& R^1\pi_{n,G\, \ast}(\End_0(\Escr_\phi))
\arrow[r] &0\\
0 \arrow[r]& R^1\pi_{n,G\, \ast}
\Omega_{\mathfrak{X}_{G}/M^{rs}_{G}}
\arrow[u, "m_{\phi}" description] \arrow[r] & 
R^1\pi_{n,G\, \ast}( \Sfrak^{-1}_{\mathfrak{X}_{G}/M^{rs}_{G}}(\Pscr)) 
\arrow[r]\arrow[u]&
\mathcal{T}_{M^{rs}_{G}/S}\arrow[r] \arrow[u, "-\phi_*" description]&0.
\end{tikzcd}
\end{equation*}
In this diagram, by Proposition \ref{conj:firstone} 
we get the isomorphism of the first two rows.
Finally, the map between the second and the third row follows from Proposition
\ref{prop:veryimportant} and Proposition \ref{eqn:veryimp}. 
Thus we get that $R^1\pi_{n,{G}*} (\mathfrak{S}^{-1}_{\mathfrak{X}_{G}/M^{rs}_{G}}(\Pscr))$ is isomorphic to the Atiyah algebra 
$\frac{1}{m_{\phi}}\phi^{\bullet}\operatorname{At}_{M^{s}_{\SL_r}/S}(\mathcal{L})$, where $\phi^{\bullet}$ denote 
the pull-back in the category of Atiyah algebras.

\section{Parabolic analog of Beilinson-Schechtman construction}\label{sec:par-atiyah}

We now extend the previous considerations to the case of parabolic bundles.
In order to analyze parabolic Atiyah algebras for families  of parabolic bundles 
on a curve $C$, we adopt notion of $\Gamma$-linearized
bundles on a Galois cover $\widehat{C}\,\longrightarrow\, C$ with Galois group $\Gamma$. 
See \cite[\S~6]{BMW1} for more details.

\subsection{Parabolic vector bundles} \label{sec:par-vector}

Let $\widetilde{\mathcal{E}}$ be a vector bundle of rank $r$ on a family of ramifed $\Gamma$-cover of curves 
$\widetilde{\pi}\,:\,\widetilde{X}\,\longrightarrow\,{T}$ ramified along $\widehat{D}$. In other words,
there is a natural projection 
$p\,:\, \widetilde{X}\,\longrightarrow\, X$ which is a
ramifed $\Gamma$-covering such that $\widetilde{\pi}\,=\,\pi\circ p$. Let
$D\,:=\,p(\widehat{D})\,\subset\, X$ be the divisor or marked points. Let $\widetilde{\mathcal{E}}$ be a
family of vector bundles on $\widetilde{X}$ which is $\Gamma$-linearized. Let $\mathcal{E}$ 
be the vector bundle on $X$ defined by the invariant pushforward of
the $\Gamma$-bundle $\widetilde{{\mathcal{E}}}$. By the discussion in
\cite[\S~7]{BMW1}, the vector bundle $\mathcal{E}$ comes equipped with a parabolic structure supported
on $D$.
Recall the trace-zero relative Atiyah sequence 
\begin{equation}\label{eqn:equiatiyah}
0 \lra\End_0(\widetilde{\mathcal{E}})
\lra\operatorname{At}_{\widetilde{X}/T}(\widetilde{\mathcal{E}})
\lra\mathcal{T}_{\widetilde{X}/T}\lra 0.
\end{equation} 
By Seshadri \cite{SeshadriI}, we can identify the sheaf of parabolic endomorphism $\Par(\widetilde{\Ecal})$ with $p_{*}^{\Gamma}\End_{0}(\widetilde{\Ecal)}$. With this set-up we can consider the  {\em parabolic Atiyah algebra} with the following fundamental exact  sequence
$$0 \lra {\Par_0(\widetilde{\Ecal})\atop \cong p_*^{\Gamma}(\End_0(\widetilde{\Ecal})}\lra {{}^{par}\At_{X/T}(\Ecal)\atop:= p_*^{\Gamma}\At_{\widetilde{X}/T}(\widetilde{\Ecal})}\lra \mathcal{T}_{X/T}(-D)\lra 0.$$
As before, $\widehat{D}$ is the relative ramification divisor in
$\widetilde{X}$; consider the log-relative traceless Atiyah sequence obtained from \eqref{eqn:equiatiyah}
\begin{equation*}
0 \lra\End_0(\widetilde{\mathcal{E}})(-\widehat{D})
\lra(\operatorname{At}_{\widetilde{X}/T}(\widetilde{\mathcal{E}}))(-\widehat{D})
\lra\mathcal{T}_{\widetilde{X}/T}(-\widehat{D})\lra 0.
\end{equation*}
Since all the objects are naturally $\Gamma$-linearized, we can apply the
invariant push-forward 
functor $p_*^{\Gamma}$ to  get the \emph{strongly parabolic Atiyah algebra} ${}^{spar}\At_{X/T}(\Ecal)$ with the fundamental exact sequence

\begin{equation*}
\label{eqn:sparabolicatiyahseq}
0\lra {\SParEnd_0(\mathcal{E})\atop \cong p_*^{\Gamma}\left(\End_0(\widetilde{\Ecal})(-\widehat{D})\right)} \lra p_*^{\Gamma}\left(
\operatorname{At}_{\widetilde{X}/T}(\widetilde{\mathcal{E}})(-\widehat{D})\right)
\lra \mathcal{T}_{X/T}(-D)\lra 0
\end{equation*}
Tensoring with $\mathcal{O}_X(D)$ we get, 
\begin{equation*}
\label{eqn:sparabolicatiyahseqnew}
0\lra \SParEnd_0(\mathcal{E})(D) \lra p_*^{\Gamma}\left(
\operatorname{At}_{\widetilde{X}/T}(\widetilde{\mathcal{E}})(-\widehat{D})\right)(D)
\lra \mathcal{T}_{X/T}\lra 0.
\end{equation*}
As in Section \ref{sec:ginzburg}, consider the dual exact sequence:
\begin{equation}\label{eqn:sparpush}
0 \lra \Omega_{{X}/T} \lra \left(  p_*^{\Gamma}\left(
\operatorname{At}_{\widetilde{X}/T}(\widetilde{\mathcal{E}})(-\widehat{D})\right)(D)\right)^{\vee}
\lra \left(\SParEnd_0(\mathcal{E})(D)\right)^{\vee} \lra 0.
\end{equation}
The trace pairing $\kappa_{\mathfrak{sl}(r)}:
\End_0(\widetilde{\mathcal{E}}) \otimes  \End_0(\widetilde{\mathcal{E}}) \rightarrow \mathcal{O}_{\widetilde{X}}$
gives an $\mathcal{O}_{\widetilde{X}}(-\widehat{D})$ valued pairing 
\begin{equation*}
\kappa_{\mathfrak{sl}(r)}:  \End_0(\widetilde{\mathcal{E}})\otimes
\End_0(\widetilde{\mathcal{E}})(-\widehat{D}) \lra \mathcal{O}_{\widetilde{X}}(-\widehat{D}).
\end{equation*}
Taking $\Gamma$-invariant push forward $p_*^{\Gamma}$ of the above exact sequence we get a map 
$$\kappa_{\mathfrak{sl}(r)}:  \ParEnd_0(\mathcal{E})\otimes \SParEnd_0(\mathcal{E})
\lra\mathcal{O}_{{X}}(-D).$$ Now by multiplying by $\mathcal{O}_{X}(D)$ on both sides we get the following:
\begin{proposition}
	\label{prop:parabolictrace}
	The trace induces a nondegenerate pairing
	$$\kappa_{\mathfrak{sl}(r)}: \ParEnd_0({\mathcal{E}})\otimes
	\SParEnd_0({\mathcal{E}})(D)\ \lra \mathcal{O}_X$$ which
	identifies $(\SParEnd_0({\mathcal{E}})(D))^\vee\cong \ParEnd_0({\mathcal{E}})$. 
\end{proposition}
We pull back the sequence in \eqref{eqn:sparpush} via the map $\nu_{\mathfrak{sl}(r)}$ to get the following quasi-Lie algebra which we denote by 
${}^{spar}\widetilde{\At}_{X/T}(\Ecal) $
\begin{equation}\label{eqn:spardual}
\begin{tikzcd}[column sep=small]
0 \arrow[r] & \Omega_{X/T} \arrow[r]& \left(  p_*^{\Gamma}\left(
\operatorname{At}_{\widetilde{X}/T}(\widetilde{\mathcal{E}})(-\widehat{D})\right)(D)\right)^{\vee}
\arrow[r] & (\SParEnd_0(\mathcal{E})(D))^{\vee}  \arrow[r]& 0\\
0 \arrow[r] & \Omega_{X/T} \arrow[r]\arrow[u, equal] &
{}^{spar}\widetilde{\At}_{X/T}(\mathcal{E})
\arrow[r]\arrow[u,"\cong"]& \ParEnd_0({\mathcal{E}}) \arrow[r]\arrow[u,"\nu_{\mathfrak{sl}(r)}"] &0.
\end{tikzcd}
\end{equation}
\subsubsection{Parabolic Bloch-Esnault complex} \label{sec:par-be}

By the construction of Beilinson-Schechtman \cite{BS88}, we get an exact sequence of sheaves: 
\begin{equation*}
0 \lra \End(\widetilde{\mathcal{E}})\otimes \Omega_{\widetilde{X}/T}\lra
{}^{tr}\widetilde{\mathcal{A}}_{\widetilde{X}/T}(\widetilde{\mathcal{E}})^{-1}\lra
\operatorname{At}_{\widetilde{X}/T}(\widetilde{\mathcal{E}})\lra 0.
\end{equation*}
Recall that since $\widetilde{X} \rightarrow T$ is a $\Gamma$-cover of
curves, and now assume that $\widetilde{\mathcal{E}}$ is $\Gamma$-equivariant. 
This implies that all terms in the above exact sequence have
a $\Gamma$-action. Taking $\Gamma$-invariant pushforward we get the following:
\begin{equation}\label{eqn:pushforwardoftrace}
0 \lra p_*^{\Gamma}(\End(\widetilde{\mathcal{E}})\otimes
\Omega_{\widetilde{X}/T})\lra
p_*^{\Gamma}({}^{tr}\widetilde{\mathcal{A}}_{\widetilde{X}/T}(\widetilde{\mathcal{E}})^{-1})\lra
p_*^{\Gamma}(\operatorname{At}_{\widetilde{X}/T}(\widetilde{\mathcal{E}}))\lra 0.
\end{equation}
Here $\At_{\widetilde{X}/T}(\widetilde{\Ecal})$ is the relative Atiyah algebra of the bundle of the vector bundle $\widetilde{\Ecal}$. 
The last term of the above is the \emph{parabolic Atiyah algebra}
$\ParAt_{X/T}({\mathcal{E}})$.
Pulling back the exact sequence in \eqref{eqn:pushforwardoftrace} via the natural inclusion
$\ParEnd_0({\mathcal{E}}) \hookrightarrow  \ParAt({\mathcal{E}})$,
we get an exact sequence: 
\begin{equation}
\label{eqn:pushbackoftrace}
0 \lra p_*^{\Gamma}(\End(\widetilde{\mathcal{E}})\otimes
\Omega_{\widetilde{X}/T})\lra 
{}^{par}\widetilde{\mathcal{B}}_{0,{X}/T}^{-1}({\mathcal{E}})\lra
\ParEnd_0({\mathcal{E}})\lra 0.
\end{equation}

Taking invariant pushforward with respect to the trace of an endomorphism  $\operatorname{tr}: \mathcal{E}nd(\widetilde{\mathcal{E}}) \rightarrow
\mathcal{O}_{\widetilde{X}}$,
we get another map $\operatorname{tr}:p_*^{\Gamma}\End (\widetilde{\mathcal{E}}\otimes
\Omega_{\widetilde{X}/T}) \lra p_*^{\Gamma}\Omega_{\widetilde{X}/T}.$

Since $\Omega_{\widetilde{X}/T}\cong p^*\Omega_{{X}/T}\otimes \mathcal{O}(\widehat{D})$, 
where $\widehat{D}$ is the relative ramification $p: \widetilde{X} \rightarrow X$. We can identify $p_*^{\Gamma}(\Omega_{\widetilde{X}/T})\cong \Omega_{X/T}$.  This in turn gives a map 
\begin{equation}\label{eqn:trace}
\operatorname{tr}: p_*^{\Gamma}(\End (\widetilde{\mathcal{E}})\otimes
\Omega_{\widetilde{X}/T}) \lra {\Omega}_{X/T}.
\end{equation}
Taking pushforward of the exact sequence in \eqref{eqn:pushforwardoftrace} via the parabolic trace in \eqref{eqn:trace}, we get the following exact sequence: 
\begin{equation}\label{eqn:parabolicdgla}
0 \lra \Omega_{{X}/T} \lra
{}^{par}\mathcal{B}_{0,X/T}^{-1}({\mathcal{E}})\lra \ParEnd_0({\mathcal{E}})
\lra 0.
\end{equation}
We can summarize the above discussion in the following commutative diagram:
\begin{equation*}
\label{eqn:commutativediagram}
\begin{tikzcd}[column sep=small]
0\arrow[r] & p_*\left(\End(\widetilde{{\mathcal{E}}})\otimes \Omega_{\widetilde{X}/T}\right) \arrow[r] 	&p_*({}^{tr}\widetilde{\mathcal{A}}_{\widetilde{X}/T}(\widetilde{{\mathcal{E}}})^{-1}) \arrow[r] 
& p_*\left(\operatorname{At}_{\widetilde{X}/T}(\widetilde{{\mathcal{E}}})\right)  \arrow[r] &0\\
0\arrow[r] & p_*^{\Gamma}(\End(\widetilde{{\mathcal{E}}})\otimes \Omega_{\widetilde{X}/T}) \arrow[r] \arrow[u, hook] 	&p_*^{\Gamma}({}^{tr}\widetilde{\mathcal{A}}_{\widetilde{X}/T}(\widetilde{{\mathcal{E}}})^{-1}) \arrow[r] \arrow[u, hook] 
& \ParAt_{\widetilde{X}/T}({{\mathcal{E}}})  \arrow[u, hook] \arrow[r] &0\\
0\arrow[r] & p_*^{\Gamma}(\End(\widetilde{{\mathcal{E}}})\otimes \Omega_{\widetilde{X}/T}) \arrow[u,equal]\arrow[r]\arrow[d, "\operatorname{tr}"]	&{}^{par}\widetilde{\mathcal{B}}_{\widetilde{X}/T}^{-1}({\mathcal{E}}) \arrow[r] \arrow[d, "\operatorname{tr}_*"]\arrow[u,hook]
& \ParEnd_0({\mathcal{E}})  \arrow[r] \arrow[d,equal] \arrow[u,hook]&0\\
0\arrow[r] & \Omega_{{X}/T} \arrow[r]	&{}^{par}\Bcal_{0,_{{X}/T}}^{-1}({\mathcal{E}}) \arrow[r]
& \ParEnd_0({\mathcal{E}})  \arrow[r] &0.
\end{tikzcd}
\end{equation*}
Recall from Section 2.1.1.1 in Beilinson-Schechtman \cite[Lemma (a)]{BS88}, that there is a residue pairing
$$\widetilde{\operatorname{Res}}: \Omega_{\widetilde{X}/T}\boxtimes
\Omega_{\widetilde{X}/T}(3\D)\lra \mathcal{O}_{\widetilde{X}}.$$
The
following theorem connects the parabolic Ginzburg dgla defined above to the
quasi Lie algebra ${}^{spar}\widetilde{\At}_{X/T}(\mathcal{E})$ defined by \eqref{eqn:spardual}.
\begin{theorem}
	\label{thm:maindgla}
	There is an isomorphism induced by invariant push-forward of the
	residue pairing between the quasi-Lie algebras
	${}^{spar}\widetilde{\At}_{X/T}(\mathcal{E})$  and ${}^{par}
	\mathcal{B}_{0,X/T}^{-1}(\mathcal{E})$ which induces an isomorphism of exact sequences: 
	\begin{equation*}
	\begin{tikzcd}
	0\arrow[r] & \Omega_{{X}/T} \arrow[r]\arrow[d,equal]
	&{}^{spar}\widetilde{\At}_{X/T}(\mathcal{E}) \arrow[r]\arrow[d,"\cong"]
	& \ParEnd_{0}({\mathcal{E}})  \arrow[r]\arrow[d,"-\id","\cong"'] &0 \\
	0\arrow[r] & \Omega_{{X}/T} \arrow[r]	&{}^{par}\Bcal_{0,X/T}^{-1}({\mathcal{E}}) \arrow[r]
	& \ParEnd_0({\mathcal{E}})  \arrow[r] &0.
	\end{tikzcd}
	\end{equation*}
\end{theorem}
\begin{proof}Let $\widetilde{\mathcal{E}}':=\widetilde{\mathcal{E}}\otimes \Omega_{\widetilde{X}/T}$, as in Beilinson-Schechtman, the sheaf of first order differential operators $\mathcal{D}^{\leq 1}( \WE)$ can be identified with $\frac{\BOX(2\D)}{\BOX}$. There is a natural exact sequence of sheaves 
	\begin{equation*}
	0 \lra \frac{\BOX(\D)}{\BOX} \lra\frac{\BOX(2\D)}{\BOX}
	\lra\frac{\BOX(2\D)}{\BOX(\Delta)} \lra 0. 
	\end{equation*}
	Moreover, all the objects of the above exact sequence are $\Gamma$-linearized. In particular apply the invariant pushforward functor we get the following commutative diagram: 
	
	\begin{equation*}
	0 \lra p_*^{\Gamma}\frac{\BOX(\D)}{\BOX}  \lra p_*^{\Gamma}
	\frac{\BOX(2\D)}{\BOX} \lra
	p_*^{\Gamma}\frac{\BOX(2\D)}{\BOX(\Delta)} \lra 0.	
	\end{equation*} 
	Similarly we have another short exact sequence  from \cite{BS88} that is used in constructing the trace complex. 
	\begin{equation*}
	0 \lra \frac{\BOX}{\BOX(-\D)} \lra \frac{\BOX(2\D)}{\BOX(-\D)} \lra
	\frac{\BOX(2\D)}{\BOX} \lra 0.
	\end{equation*}
	Pulling  back the  exact sequence above by the inclusion
	$\End(\WE)\cong \frac{\BOX(\Delta)}{\BOX} \hookrightarrow \frac{\BOX(2\D)}{\BOX}.$
	we get a short exact sequence 
	\begin{equation*}
	0 \lra \frac{\BOX}{\BOX(-\D)}\lra \frac{\BOX(\D)}{\BOX(-\D)} \lra
	\frac{\BOX(\D)}{\BOX}\lra 0.
	\end{equation*}
	Taking invariant pushforward functor we get the following commutative diagram: 
	\begin{equation*}
	\begin{tikzcd}
	0 \arrow[r] & p_*^{\Gamma}\left(\frac{\BOX}{\BOX(-\D)}\right)\arrow[r] & p_*^{\Gamma}\left(\frac{\BOX(2\D)}{\BOX(-\D)}\right) \arrow[r] &p_*^{\Gamma}\left(\frac{\BOX(2\D)}{\BOX}\right) \arrow[r] &0\\
	0 \arrow[r] & p_*^{\Gamma}\left(\frac{\BOX}{\BOX(-\D)}\right)\arrow[r] \arrow[u, equal] & p_*^{\Gamma}\left(\frac{\BOX(\D)}{\BOX(-\D)}\right) \arrow[r] \arrow[u, hook] &p_*^{\Gamma}\left(\frac{\BOX(\D)}{\BOX}\right)\arrow[u, hook] \arrow[r] &0.
	\end{tikzcd}
	\end{equation*}

	Now there is a natural  nondegenerate bilinear form: 
	\begin{equation}\label{eqn:formres1}
	\langle\, ,\, \rangle:  \frac{\BOX(\D)}{\BOX(-\D)}\otimes
	\frac{\BOX(2\D)}{\BOX} \lra \mathcal{O}_{\widetilde{X}}.
	\end{equation}
	The above form $\langle \,  ,\, \rangle$ vanishes identically restricted to $ \frac{\BOX}{\BOX(-\D)}\otimes \frac{\BOX(\D )}{\BOX}$ and hence it descends to a form 
	\begin{equation}\label{eqn:formres2}
	\langle\,  ,\, \rangle: \frac{\BOX}{\BOX(-\D)}\otimes \frac{\BOX}{\BOX(\D)}
	\lra \mathcal{O}_{\widetilde{X}}, \ \mbox{and} \ 
	\langle\, ,\, \rangle: \frac{\BOX(\D)}{\BOX}\otimes \frac{\BOX(\D)}{\BOX}
	\lra \mathcal{O}_{\widetilde{X}}.
	\end{equation}
	Under the canonical identifications of $\frac{\BOX(\D)}{\BOX}\cong \End(\WE)$ we get
	\begin{align}
	\frac{\BOX}{\BOX(\D)}&\cong \End(\WE)\otimes \mathcal{T}_{\widetilde{X}/T}, \ \mbox{and}\
	\frac{\BOX}{\BOX(-\D)}\cong \End(\WE)\otimes \Omega_{\widetilde{X}/T}.
	\end{align}
	Lemma B.2.8 in \cite{BBMP20} shows that the bilinear form $\langle\,
	,\, \rangle $ in \eqref{eqn:formres1} and \eqref{eqn:formres2} can be identified as 
	\begin{align}\label{eqn:pushform1}
	-\kappa_{\mathfrak{gl}(r)}:\End(\WE)\otimes \End(\WE)\lra \mathcal{O}_{\widetilde{X}}\ \mbox{and} \
	\kappa_{\mathfrak{gl}(r)}:\End(\WE)\otimes
	\mathcal{T}_{\widetilde{X}/T}\otimes \End(\WE)\otimes
	\Omega_{\widetilde{X}/T}\lra \mathcal{O}_{\widetilde{X}},
	\end{align}where $\kappa_{\mathfrak{gl}(r)}$ is the trace of the product of the two endomorphisms. 
	Now the invariant pushforward of $\langle\, ,\, \rangle$ induces a nondegenerate $\mathcal{O}_X(-D)$ valued form  
	\begin{equation*}\label{eqn:invpushform0}
	\langle\, ,\, \rangle: p_*^{\Gamma}\frac{\BOX(\D)}{\BOX(-\D)}\otimes
	p_*^{\Gamma}\left(\frac{\BOX(2\D)}{\BOX}(-\widehat{D})\right)\lra p_*^{\Gamma}\mathcal{O}_{\widetilde{X}}(-\widehat{D})\cong \mathcal{O}_X(-D)
	\end{equation*} which gives the following nondegenerate form 
	\begin{equation*}\label{eqn:invpushform1}
	\langle\, ,\,  \rangle: p_*^{\Gamma}\frac{\BOX(\D)}{\BOX(-\D)}\otimes
	\left(p_*^{\Gamma}\left(\frac{\BOX(2\D)}{\BOX}(-\widehat{D})\right)\right)(D)\lra \mathcal{O}_X
	\end{equation*}
	that restricts to 
	\begin{align}\label{eqn:invpushform2}
	\langle\, ,\, \rangle : p_*^{\Gamma}\left(\frac{\BOX(\D)}{\BOX} \right)
	\otimes \left(p_*^{\Gamma}\left(\frac{\BOX(\D)}{\BOX}(-\widehat{D})
	\right)\right)(D)\lra \mathcal{O}_X.
	\end{align}
	\begin{align}\label{eqn:invpushform3}
	\langle\, ,\, \rangle: p_*^{\Gamma}\left(\frac{\BOX}{\BOX(-\D)} \right)
	\otimes \left(p_*^{\Gamma}\left(\frac{\BOX}{\BOX(\D)}(-\widehat{D})
	\right)\right)(D)\lra \mathcal{O}_X.
	\end{align}
	The identification of parabolic and strongly parabolic endomorphism as
	invariant pushforward and invariants pushforward of the identification
	in \eqref{eqn:pushform1} tell us that the form $\langle\, ,\, \rangle $ in \eqref{eqn:invpushform2} can be identified with 
	\begin{equation*}
	-\kappa_{\mathfrak{gl}(r)}: \ParEnd(\WE)\otimes \SParEnd(\WE)(D)\lra \mathcal{O}_X.
	\end{equation*}
	Similarly the form in \eqref{eqn:invpushform3} can be identified with 
	\begin{equation*}
	\langle\, ,\,  \rangle: p_*^{\Gamma}\left(\End(\WE)\otimes
	\Omega_{\widetilde{X}/T}\right)\otimes
	\left(p_*^{\Gamma}\left(\End(\WE)\otimes
	\mathcal{T}_{\widetilde{X}/T}(-\widehat{D})\right)\right)(D)\lra \mathcal{O}_X
	\end{equation*}
	Now the identification of 
	$\left(p_*^{\Gamma}\left(\End(\WE)\otimes \mathcal{T}_{\widetilde{X}/T}(-\widehat{D})\right)\right)(D)\cong \mathcal{T}_{X/T}$ gives the following commutative 
	\begin{equation*}
	\begin{tikzcd}
	\langle\, ,\, 
	\rangle:  \left(p_*^{\Gamma}\left(\End(\WE)\otimes \Omega_{\widetilde{X}/T}\right)\right)\otimes \left(p_*^{\Gamma}\left(\End(\WE)\otimes \mathcal{T}_{\widetilde{X}/T}(-\widehat{D})\right)\right)(D)\arrow[r] \arrow[d, "\cong"] &\mathcal{O}_X\arrow[d, equal] \\ 
	\Omega_{{X}/T} \otimes \mathcal{T}_{X/T}\arrow[r] & \mathcal{O}_X.
	\end{tikzcd}
	\end{equation*}
	This induces an isomorphism of
	${}^{par}\mathcal{B}_{0,X/T}^{-1}(\mathcal{E})$ with
	${}^{spar}\widetilde{\At}_{X/T}(\mathcal{E})$ 
	that restricts to $\id$ on $\Omega_{X/T}$ and $-\id$ on $\Par_0(\Ecal)$. 
\end{proof}

The following proposition connects $R^1\pi_*$ of the parabolic Ginzburg dgla with $R^1\widetilde{\pi}_*$ of the dgla constructed by Bloch-Esnault.
\begin{proposition}\label{prop:comparingdgla}
	There are inclusion maps $\Omega_{X/T} \rightarrow
	p_*\Omega_{\widetilde{X}/T}$ and $\ParEnd_{0}(\mathcal{E})\cong
	p_*^{\Gamma}\End_{0}(\mathcal{E})\hookrightarrow p_*\End_{0}(\widetilde{{\mathcal{E}}})$
	that extend to a map of the following exact sequences:
	\begin{center}
		\begin{tikzcd}
			0 \arrow[r] & p_*\Omega_{\widetilde{X}/T} \arrow[r]&
			p_*\mathcal{B}_{0,\widetilde{X}/T}^{-1}(\widetilde{{\mathcal{E}}}) \arrow[r] &
			p_*\End_{0}(\widetilde{{\mathcal{E}}}) \arrow[r] & 0\\
			0 \arrow[r] & \Omega_{X/T}\arrow[u, hook] \arrow[r]&
			{}^{par}\mathcal{B}_{0,{X}/T}^{-1}(\mathcal{E}) \arrow[u, hook]
			\arrow[r] & \ParEnd_{0}(\mathcal{E})\arrow[u,hook]\arrow[r] &0.
		\end{tikzcd}
	\end{center}
\end{proposition}
\begin{proof}First we prove that
	$p_*^{\Gamma}(\widetilde{\Bcal}_{0,\widetilde{X}/T}^{-1}(\WE))\cong
	{}^{par}\widetilde{\mathcal{B}}_{0,{X}/T}^{-1}(\mathcal{E})$. 
	Consider the short exact sequence 
	\begin{equation*}
	0\lra p_*(\End(\WE)\otimes \Omega_{\widetilde{X}/T})\lra 
	p_*(\widetilde{\Bcal}_{0,\widetilde{X}/T}^{-1}(\WE))\lra p_*(\End_0(\WE))\lra 0.
	\end{equation*}
	Taking invariants with respect to $\Gamma$, we get 
	\begin{equation*}
	\begin{tikzcd}[column sep=small]
	0 \arrow[r] & p_*(\End(\WE)\otimes\Omega_{\widetilde{X}/T}) \arrow[r] & 
	p_*\widetilde{\Bcal}_{0,\widetilde{X}/T}^{-1}(\WE)\arrow[r] & p_*(\End_0(\WE))\arrow[r]&0\\
	0 \arrow[r] &
	p_*^{\Gamma}(\End(\WE)\otimes\Omega_{\widetilde{X}/T})\arrow[d,equal]
	\arrow[r]\arrow[u, hook] &
	p_*^{\Gamma}(\widetilde{\Bcal}_{0,\widetilde{X}/T}^{-1}(\WE))\arrow[r]\arrow[u,
	hook]\arrow[d,hook] & p_*^{\Gamma}(\End_{0}(\WE))\arrow[r]\arrow[u, hook]\arrow[d,hook]&0\\
	0\arrow[r]&	p_*^{\Gamma}(\End(\WE)\otimes\Omega_{\widetilde{X}/T})\arrow[r]&p_*^{\Gamma}({}^{tr}\widetilde{\mathcal{A}}_{\widetilde{X}/T}(\WE)^{-1})\arrow[r]& p_*^{\Gamma}(\operatorname{At}_{\widetilde{X}/T}(\WE)) \arrow[r] &0.
	\end{tikzcd}
	\end{equation*}
	The inclusion of the second row into the third follows from the invariant pushforward of the first two
	rows of the diagram \eqref{eqn:3rowdiagram}. 
	Since both ${}^{par}\widetilde{\Bcal}_{0,X/T}^{-1}(\mathcal{E})$
	and $p_*^{\Gamma}(\widetilde{\Bcal}_{0,\widetilde{X}/T}^{-1}(\WE))$ are extensions
	of $p_*^{\Gamma}(\End_{0}(\WE))$ by $p_*^{\Gamma}(\End(\WE)\otimes \Omega_{\widetilde{X}/T})$
	obtained as a sub extension of
	$p_*^{\Gamma}(\operatorname{At}_{\widetilde{X}/T}(\WE))$ by
	$p_*^{\Gamma}(\End(\WE)\otimes \Omega_{\widetilde{X}/T})$ via the
	inclusion $p_*^{\Gamma}(\End_{0}(\WE))\hookrightarrow
	p_*^{\Gamma}(\operatorname{At}_{\widetilde{X}/T}(\WE))$, it
	follows that 
	$p_*^{\Gamma}(\widetilde{\Bcal}_{0,\widetilde{X}/T}^{-1}(\WE))\cong
	{}^{par}\widetilde{\mathcal{B}}_{0,X/T}^{-1}(\mathcal{E}).$
	Now the above  gives the following commutative  diagram:
	\begin{equation*}
	\begin{tikzcd}
	0\arrow[r]&	p_*(\End(\WE)\otimes \Omega_{\widetilde{X}/T}) \arrow[r] &
	p_*\widetilde{\Bcal}_{0,\widetilde{X}/T}^{-1}(\WE)\arrow[r] & p_*(\End_{0}(\WE)) \arrow[r] &0
	\\
	0 \arrow[r]&	p_*^{\Gamma}\left(\End(\WE))\otimes
	\Omega_{\widetilde{X}/T} \right)\arrow[r]\arrow[u, hook]
	&{}^{par}\widetilde{\mathcal{B}}_{0,\widetilde{X}/T}^{-1}(\mathcal{E})\arrow[r]\arrow[u,
	hook] & \ParEnd_{0}(\mathcal{E})\arrow[u, hook] \arrow[r] &0.
	\end{tikzcd}
	\end{equation*}
	Pushing forward with respect to the trace of endomorphism $\operatorname{tr}$, the above commutative
	diagram implies the existence of the following commutative diagram all of whose rows are short exact sequences: 
	\begin{equation*}
	\adjustbox{scale=1}{
		\begin{tikzcd}[column sep=0.1]
		p_*(\End(\WE)\otimes \Omega_{\widetilde{X}/T})
		\arrow[rr,hook]\arrow[dd,"p_*\operatorname{tr}"] &&
		p_*\widetilde{\Bcal}_{0,\widetilde{X}/T}^{-1}(\WE)\arrow[rr,twoheadrightarrow]\arrow[dd,"\operatorname{tr}_*"
		near start] && p_*(\End_{0}(\WE))\arrow[dd,equal]&\\
		&	p_*^{\Gamma}(\End(\WE))\otimes
		\Omega_{\widetilde{X}/T})\arrow[rr,crossing over,
		hook]\arrow[ul, hook]
		&&{}^{par}\widetilde{\mathcal{B}}_{0,{X}/T}^{-1}(\mathcal{E})\arrow[rr,crossing
		over, twoheadrightarrow]\arrow[lu, hook] && \ParEnd_{0}(\mathcal{E})\arrow[dd,equal]\arrow[ul, hook]\\
		p_*\Omega_{\widetilde{X}/T} \arrow[rr,hook]&&
		p_*\Bcal_{0,\widetilde{X}/T}^{-1}(\WE)\arrow[rr] && p_*(\End_{0}(\WE))&  \\
		&\Omega_{X/T}\arrow[ul, hook]\arrow[rr,hook]
		\arrow[from=uu,crossing over, "\operatorname{tr}" near start
		]&& {}^{par}\mathcal{B}_{0,{X}/T}^{-1}(\mathcal{E})\arrow[ul, hook]
		\arrow[rr,twoheadrightarrow]\arrow[from=uu,crossing over,
		"\operatorname{tr}_*" near start]&& \ParEnd_{0}(\mathcal{E}).\arrow[ul, hook]
		\end{tikzcd}
	}
	\end{equation*}
	The bottom level of the above diagram gives the required result.
\end{proof}

Recall that we have a diagram relating the families of curves parametrized by $T$:
\begin{center}
	\begin{tikzcd}
		\widetilde{X} \arrow[r, "p"] \arrow[dr, "\widetilde{\pi}"'] & X \arrow[d, "\pi"]\\
		&T
	\end{tikzcd}
\end{center}
Taking $R^1\pi_*$ of all the terms of the commutative diagram in Proposition \ref{prop:comparingdgla}, and using the fact that $R^1\widetilde{\pi}=R^1\pi\circ p_*$, we get the following proposition:

\begin{proposition}
	\label{cor:Jpetermancorollary}The following diagram is commutative:
	\begin{equation*}
	\begin{tikzcd}
	0 \arrow[r] & R^1\widetilde{\pi}_*\Omega_{\widetilde{X}/T}\simeq
	\mathcal{O}_T \arrow[r]&
	R^1\widetilde{\pi}_*(\Bcal_{0,\widetilde{X}/T}^{-1}(\widetilde{\mathcal{E}}))\arrow[r]
	& R^1\widetilde{\pi}_* \End_{0}(\widetilde{{\mathcal{E}}}) \arrow[r] & 0
	\\
	0 \arrow[r] & R^1{\pi}_*\Omega_{{X}/T}\simeq \mathcal{O}_T
	\arrow[r]\arrow[u, equal]&
	R^1{\pi}_*({}^{par}\mathcal{B}_{0,{X}/T}^{-1}({\mathcal{E}}))\arrow[r]\arrow[u,
	hook] & R^1{\pi}_* \ParEnd_{0}({{\mathcal{E}}}) \arrow[u, hook] \arrow[r] & 0.
	\end{tikzcd}
	\end{equation*}
\end{proposition}

\subsection{Parabolic Atiyah algebras and moduli of parabolic bundles}
Let $\Ccal\rightarrow S$ be a family of curves, and let $\widehat{\Ccal}\rightarrow S$ be a family of $\Gamma$-covers. 
Consider the relative family $\widehat{M}^{s}_{\SL_r}$ parametrizing the moduli space of stable $\SL_r$ bundles on 
$\widehat{\Ccal}\rightarrow S$. Let $M^{\bm\tau,s}_{\SL_r,s}$ and ${M}^{par,s}_{\SL_r}$ be the relative moduli 
spaces of stable $(\Gamma, \SL_r)$ and stable parabolic bundles on $\widehat{\Ccal}\rightarrow S$ and 
$\Ccal\rightarrow S$ respectively. Without loss of generality assume that the interior of $\widehat{M}_{\SL_r}^{ss}$ is non-empty, otherwise our main theorem is trivially true. 

Now by the discussion in \cite[\S 7.2]{BMW1}, the invariant
pushforward functor induces an isomorphism between $M^{\bm\tau,s}_{\SL_r}$ (respectively, the semi-stable moduli
space $M^{\bm\tau,ss}_{\SL_r}$)
and ${M}^{par,s}_{\SL_r}$ (respectively, ${M}^{par,ss}_{\SL_r}$). The data ${\bm \tau}$ and the covering family
$\widehat{\mathcal{C}}$ depend on the data of the parabolic weights that defines the parabolic semistability. By
\cite{Biswasduke}, we get a map 
$\phi: M^{\bm\tau,ss}_{\SL_r}\rightarrow \widehat{M}^{ss}_{\SL_r}$ which extends to
a map $\widehat{\Ccal}\times_S M^{\bm\tau,ss}_{\SL_r}\rightarrow
\widehat{\Ccal} \times_{S} \widehat{M}^{ss}_{\SL_r}$. The map $\phi$ may
not preserve stability, however $\phi$ being finite, the complement of the
inverse image $Y^{\bm \tau,s}_{\SL_r}:=\phi^{-1}\widehat{M}^{s}_{\SL_r}$ of
the stable locus has codimension at least two in $M_{\SL_r}^{\tau,s}$
provided genus of the orbifold curve $\mathscr{C}=[\widehat{C}/\Gamma]$ is
at least $2$ if $r\neq 2$ and at least $3$ if $r=2$. We refer the reader to
\cite[Lemma 8.3]{BMW1} for more details. 

We have the following diagrams that connect all the objects described above:
\begin{equation*}
\adjustbox{scale=0.9}{
	\begin{tikzcd}
	\widehat{\Ccal}\times_S \widehat{M}^{s}_{\SL_r}\arrow[ddd,"\widehat{\pi}_w"] \arrow[rrrrr,"\widehat{\pi}_n", near
	start]&&&&&\widehat{M}^{s}_{\SL_r}\arrow[ddd,"\widehat{\pi}_e"]\\
	& {\widehat{\Ccal}\times_S 
		M^{\bm\tau,s}_{\SL_r} \atop {\overset{\rotatebox{90}{$\subseteq$}}{\widehat{\Ccal}\times_S 
				Y^{\bm\tau,s}_{\SL_r}}}} \arrow[lu,shift left=3ex, "\id\times \phi"]\arrow[lu, dashed, "\id \times
	\phi"'] \arrow[rd,"p\times \cong", ]\arrow[rr,"\widetilde{\pi}_n", near
	start] & &{ 
		M^{\bm\tau,s}_{\SL_r} \atop {\overset{\rotatebox{90}{$\subseteq$}}{ 
				Y^{\bm\tau,s}_{\SL_r}}}}\arrow[rd,"\cong"]\arrow[rru,dashed,"\phi"]\arrow[rru,shift right=3ex, "\phi"']&\\
	&&{\Ccal\times_S {M}^{par,s}_{\SL_r} \atop {\overset{\rotatebox{90}{$\subseteq$}}{{\Ccal \times_SY_{\SL_r}^{par,s}}}}} && {{M}^{par,s}_{\SL_r} \atop {\overset{\rotatebox{90}{$\subseteq$}}{{Y_{\SL_r}^{par,s}}}}}\arrow[from=ll,"{\pi}_n", near start, crossing over]&&\\
	\widehat{\Ccal}\arrow[rd,equal]\arrow[rrrrr,"\widehat{\pi}_s" description] &&&&& S\arrow[from=uuu,"\widehat{\pi}_e", crossing over]\\
	& \widehat{\Ccal}\arrow[from=uuu,"\widetilde{\pi}_w", crossing over, near start]\arrow[rr,"\widetilde{\pi}_s", near start]\arrow[rd,"p"]& &S \arrow[rru,equal]\arrow[rd,equal] \arrow[from=uuu,"\widetilde{\pi}_{e}"]&\\
	&&\Ccal\arrow[rr,,"{\pi}_s", near start] \arrow[from=uuu,"{\pi}_w", crossing over, near start]&& S\arrow[from=uuu,"{\pi}_e",crossing over, near end]&
	\end{tikzcd}}
\end{equation*}
The rational maps are regular over $Y_{\SL_r}^{\bm \tau,s}$ which will be also denoted by the same notation.
The image of $Y_{\SL_r}^{\bm \tau,s}$ under the invariant pushforward isomorphism of $M_{\SL_r}^{\bm \tau,ss}
\cong M_{\SL_r}^{par,ss}$ will be denoted by $Y_{\SL_r}^{par,s}$. By definition $ Y_{\SL_r}^{par,s}
\hookrightarrow  M_{\SL_r}^{par,s}$.  

Let $\widehat{\mathcal{E}}$ be the universal bundle (which exist in the \'etale toplogy) on $\widehat{\Ccal}\times_S\widehat{M}^{ss}_{\SL_r}$ and 
$\widetilde{ \mathcal{E}}$ be its pull-back to $\widehat{\Ccal}\times_S M^{\bm\tau,s}_{\SL_r}$. We denote by 
$\Ecal$ the universal parabolic bundle which we can assume to exist without loss of generality (see Remark 
\ref{rem:adjoint}). As in the diagram let $\widetilde{\pi}_n: \widehat{C} \times_S Y_{\SL_r}^{\bm \tau,s} \lra Y_{\SL_r}^{\bm \tau,s}$ denote the projection and similarly consider the projection $\widehat{\pi}_n: \widehat{C} \times_S \widehat{M}_{\SL_r}^{s} \lra \widehat{M}_{\SL_r}^{s}$

Let $\mathcal{L}$ be the determinant of cohomology line bundle on
$\widehat{M}^{s}_{\SL_r}$. Now, as before, combining the results of
Baier-Bolognesi-Martens-Pauly \cite{BBMP20}, Beilinson-Schechtman
\cite{BS88}, Bloch-Esnault \cite{BE}, and Sun-Tsai \cite{ST}, we get an isomorphism of the Atiyah algebras $\At_{\widehat{M}^{s}_{\SL_r}/S}(\mathcal{L}^{-1})$ with $R^1\widehat{\pi}_{n*}
(\Bcal_{0,\widehat{\mathcal{C}}\times_S\widehat{M}_{\SL_r}^s/\widehat{M}^s_{\SL_r}}^{-1}(\widehat\Ecal))$ which makes the diagram of  fundamental
sequences of Atiyah algebras commute: 
\begin{equation*}
\begin{tikzcd}
0 \arrow[r] & \mathcal{O}_{\widehat{M}^{s}_{\SL_r}} \arrow[r] \arrow[d, equal] & \At_{\widehat{M}^{s}_{\SL_r}/S}(\mathcal{L}^{-1}) \arrow[d, "\cong"]\arrow[r] & \mathcal{T}_{\widehat{M}^{s}_{\SL_r}/S}\arrow[d,equal] \arrow[r] &0 \\ 
0 \arrow[r]& \mathcal{O}_{\widehat{M}^{s}_{\SL_r}} \arrow[r] &
R^1\widehat{\pi}_{n*}(\mathcal{B}_{0,\widehat{\mathcal{C}}\times_S\widehat{M}_{\SL_r}^s/\widehat{M}^s_{\SL_r}}^{-1}(\widehat \Ecal)) \arrow[r] & \mathcal{T}_{\widehat{M}^{s}_{\SL_r}/S} \arrow[r] &0.
\end{tikzcd}
\end{equation*}
Pulling back by $\phi$  we get an
isomorphism of $\phi^*\At_{\widehat{M}^{s}_{\SL_r}/S}(\mathcal{L})$ and $\phi^*R^1\widehat{\pi}_{n*}(\Bcal_{0,\widehat{\mathcal{C}}\times_S\widehat{M}_{\SL_r}^s/\widehat{M}^s_{\SL_r}}^{-1}(\widehat{\Ecal}))$. Moreover the base change theorems implies that the later is isomorphic to
$R^1\widetilde{ \pi}_{n*}(\mathcal{B}_{0,\widehat{\mathcal{C}}\times_SY^{\bm \tau,s}_{\SL_r}/Y^{\bm \tau ,s }_{\SL_r}}^{-1}(\widehat{\Ecal}))$. (Observe that $\phi^*$ of an Atiyah algebra may not be an Atiyah algebra.) We have the following result.


\begin{theorem} \label{thm:atiyah-det-be}
	There is an  isomorphism of the relative Atiyah algebras
	$\At_{M^{\bm  \tau,s}_{\SL_r}/S}(\phi^*\mathcal{L}^{-1})$ with
	$R^1\pi_{n*}({}^{par}\mathcal{B}_{0,\mathcal{C}\times_SM^{par,s}_{\SL_r}/M^{par,s}_{\SL_r}}^{-1}(\Ecal))$ that restricts to the 
	identity map on $\Omega_{M^{par,s}_{\SL_r}/S}$.
\end{theorem}

\begin{proof}
	Applying Proposition
	\ref{cor:Jpetermancorollary} with $\widetilde{X}=\widehat{\Ccal}\times_S
	Y^{\bm\tau,s}_{\SL_r}$,
	$X=\Ccal\times_S Y^{\bm\tau,s }_{\SL_r}$, and $T=Y^{\bm\tau,s}_{\SL_r}$, we get an isomorphism between
	$\operatorname{At}_{Y_{\SL_r}^{\bm \tau,s}/S}(\phi^*\mathcal{L})$ and $R^1\pi_{n*}{}^{par}\mathcal{B}_{{0,\mathcal{C}\times_SY^{\bm \tau,s}_{\SL_r}/Y^{\bm \tau,s}_{\SL_r}}}^{-1}(\Ecal)$ over $Y^{par,s}_{\SL_r}$.
	Now both these sheaves $\operatorname{At}_{Y_{\SL_r}^{\bm \tau,s}/S}({\phi^*\mathcal{L}^{-1}})$ and
	$R^1\pi_{n*}({}^{par}\mathcal{B}_{0,\mathcal{C}\times_SY^{\bm \tau,s}_{\SL_r}/Y^{\bm \tau,s}_{\SL_r}}^{-1}(\Ecal))$ are locally free
	(hence reflexive) and extend over $M^{par,s}_{\SL_r}$. Since they are
	isomorphic on an open subset whose complement has codimension at least two, 
	the isomorphism actually extends to all of $M_{\SL_r}^{par,s}$. 
\end{proof}

\begin{remark}
	The proof of Theorem \ref{thm:atiyah-bott} applies in this
	parabolic setting as well. The fact that the Atiyah-Bott-Narasimhan-Goldman symplectic
	form is in the class of $\phi^\ast\Lcal$ is one of the results of
	\cite{DW97} (see also
	\cite{BiswasRaghavendra}).
\end{remark}

\subsection{Parabolic $G$-bundles} \label{sec:par-principal}

\subsubsection{Parabolic bundles}

Let $p:\widehat C\to C$ be a ramified covering with Galois group
$\Gamma$, so $C=\widehat C/\Gamma$. Let $D\subset C$ and $\widehat
D\subset\widehat C$ denote the branching loci in $C$ and $\widehat C$ respectively. 
Let $\widehat \pi : \widehat P\to \widehat C$ be a
$\Gamma$-principal $G$-bundle, i.e.,\ $\widehat P$ is a principal
$G$-bundle, and  there is a representation
$ \Gamma\to \Aut(\widehat P): \gamma\mapsto F_\gamma$ such that the actions of $\Gamma$
and $G$ on $\widehat P$	commute.
For an open set $U\subset C$, define
$$
\aut_\Gamma(\widehat P)(U)=\left\{ X\in \aut(\widehat P)(p^{-1}(U)) \mid 
X(F_\gamma(p))=(F_\gamma)_\ast (X(p))\ ,\ \forall\ \gamma\in
\Gamma\right\}.
$$
This presheaf defines an $\Ocal_C$-coherent sheaf.
Notice that for $X\in \aut_\pi(\widehat P)$, the vector field
$\widehat v$ on $\widehat C$ is $\Gamma$-invariant, and so descends
to a vector field on $C$ that vanishes along the ramification
divisor $D\subset C$. Hence, we have
a map $\aut_\Gamma(\widehat P)\to TC(-D)$. 

Let $P$ be a parabolic ${G}$-bundle on $C$ and let
$\widehat{P}$ be a $(\Gamma, {G})$-bundle on $\widehat{C}$
such that $P$ and $\widehat{P}$ are related by invariant pushforward. 
\begin{definition}We define the parabolic Atiyah algebra
	${}^{par}\At(P)$  of $P$ to be $\aut_{\Gamma}(\widehat{P})$. 
\end{definition}

Suppose $\widehat P$ is the frame bundle of a vector bundle
$\widehat \Ecal\to \widehat C$, and let $\Ecal\to C$ denote
the sheaf of invariant sections of $\widehat \Ecal$.  
The parabolic Atiyah algebra $\ParAt(\Ecal)$
is $\aut_\Gamma(\widehat P)$.  
We wish to describe the kernel  $\ad_\Gamma(\widehat P)$ as a subsheaf of
$\aut_\Gamma(\widehat P)\to TC(-D)$. Let $\widehat U\subset \widehat C$ be a neighborhood of 
$w_0\in \widehat D$, and let   
$\Gamma_0\subset \Gamma$ be the isotropy group of $w_0$. We assume that
$\Gamma_0$ stabilizes $\widehat U$ and only $w_0\in \widehat U$ has nontrivial
isotropy.  Choose a section $s$ of $\widehat P$ over $\widehat U$. 
For each $\gamma\in \Gamma_0$ there is $\rho_\gamma: \widehat U\to G$
defined by: $F_\gamma(s(w))=s(\gamma w)\rho_\gamma(w)$. 

\begin{definition}
	Define 
	$
	\ad_\Gamma(\widehat P)(U) =\left\{ f\in \ad(\widehat P)(\widehat U) \mid
	f(s(\gamma w))=\Ad_{\rho_\gamma(w)}f(s(w))\right\}.
	$
\end{definition}
The following is straightforward.
\begin{proposition}
	The above definition is independent of the choice of
	section $s$.
\end{proposition}

Extending the definition for neighborhoods at each of the
branch points defines $\ad_\Gamma(\widehat P)$ globally. 
We also note that the inclusion map $\ad(\widehat P)\to
\aut(\widehat P)$ restricts to an inclusion map $\ad_\Gamma(\widehat P)\to
\aut_\Gamma(\widehat P)$.
Indeed, for $f\in \ad_\Gamma(\widehat P)$,  
it suffices to check the condition on the vector field $X=f^\sharp$
along the section $s$, and this follows from the
equivariance of $F_\gamma$:
$$
f(F_\gamma(s(w))=f(s(\gamma
w)\rho_\gamma(w))=\Ad_{\rho_\gamma(w)^{-1}}f(s(\gamma
w))=f(s(w)).
$$
Hence we have the following short exact sequence of sheaves of Lie algebras on $C$: 
\begin{equation*}
0\lra \ad_\Gamma(\widehat{P})\lra \aut_{\Gamma}(\widehat{P}) \lra TC(-D)\lra 0.
\end{equation*}

Suppose that $\widehat P$ is the frame bundle of a vector bundle
$\widehat \Ecal\to \widehat C$, and let $\Ecal\to C$ denote
the sheaf of invariant sections of $\widehat \Ecal$.  
The parabolic endomorphism bundle of $\Ecal$, namely
$\ParEnd(\Ecal)$,
is $\ad_\Gamma(\widehat P)$.  
We refer to the above sequence as the {\em fundamental sequence} for parabolic Atiyah algebras. 
From the  discussion above,  we have an exact sequence
\begin{equation*}
0\,\lra\, \ParEnd(\Ecal)\,\lra\ \ParAt(\Ecal)\lra TC(-D)\lra 0.
\end{equation*}
We note that this is nothing but the $\Gamma$-invariant
push-forward of the Atiyah algebra exact
sequence on $\widehat C$.

\subsubsection{Strongly parabolic Atiyah algebras}

Next, we  define a ``strongly parabolic'' version of this
construction.  Set
\begin{align*}
\aut(\widehat P)(-\widehat{D})&:=\{ X \in \aut(\widehat P) \mid X(p)=0 \text{ for
} \widehat\pi(p)\in \widehat D\}. \\
\ad(\widehat P)(-\widehat D)&:=\{ f \in \ad(\widehat P) \mid f(p)=0 \text{ for
}\widehat \pi(p)\in \widehat D\}. \\
\end{align*}
Then we have the restricted short exact sequence
\begin{equation*}
0\lra \ad(\widehat P)(-\widehat D)\lra \aut(\widehat P)(-\widehat D)\lra T\widehat
C(-\widehat D)\lra 0.
\end{equation*}
Now if $\widehat{P}$ is the frame bundle of $\widetilde{E}$, then by
Seshadri's correspondence (see Appendix B), the $\Gamma$-invariant part
$\ad_\Gamma(\widehat P(-\widehat D))$ 
is identified with the \emph{strongly} parabolic 
endomorphisms $\SParEnd(P)$. 
Motivated by this observation,  we  have:

\begin{definition}
	The strongly parabolic
	Atiyah algebra ${}^{spar}\At(P)$
	of a $\Gamma$-linearized
	principal bundle $\widehat{P}$ on $\widehat{C}$ is defined to be the
	$\Gamma$-invariant part of $\aut(\widehat P)(-\widehat D)$. 
	
	If $\widehat{P}$ is the frame bundle 
	of a vector bundle $\widehat{\Ecal}$, then we denote it by
	$\SParAt(\Ecal):=\aut_\Gamma(\widehat P)$. Noting that 
	$p_\ast(T\widehat C(-\widehat D))^\Gamma=TC(-D)$, 
	we have the following short exact sequence on $C$
	$$
	0\lra \SParEnd(\Ecal)\lra \SParAt(\Ecal)\lra TC(-D)\lra 0.
	$$
\end{definition} 

\subsubsection{Determinant line bundle for parabolic $G$-bundles}\label{sec:par-det}

Let $\widehat{\Pcal}$ be a family of $(\Gamma,G)$-bundles
parametrized by $T$  as in the previous section, and let $\Pcal$ be the family of parabolic $G$-bundles obtained by applying the invariant pushforward functor.

Consider the relative parabolic Atiyah algebra $\ParAt_{{X}/T}(\Pcal):=
p_{*}^{\Gamma}(\At_{\widetilde{X}/T} (\widehat{\Pcal})$ and the strongly parabolic Atiyah algebra
$\SParAt_{{X}/T}({\Pcal})=p_*^{\Gamma}(\At_{\widetilde{X}/T}(\widehat{\Pcal})(-\widehat{D}))$. As in
the case of parabolic vector bundles, they fit in the following fundamental exact sequences 
\begin{align}
\begin{split} \label{eqn:fund-atiyah}
0 \lra \ParEnd(\Pcal) \lra \ParAt_{{X}/T}(\Pcal)\lra
\mathcal{T}_{{X}/T}(-D)\lra 0,\\
0 \lra\SParEnd(\Pcal)
\lra\SParAt_{{X}/T}(\Pcal)\lra\mathcal{T}_{{X}/T}(-D)\lra 0,
\end{split}
\end{align}
where $\ParEnd(\Pcal)$ (respectively, $\SParEnd(\Pcal)$) denote the parabolic (respectively, strongly parabolic)
endomorphism bundle of $\Pcal$. As in the case of parabolic vector bundles we get the following quasi-Lie algebra 
\begin{equation}\label{eqn:quasiLiealgebra}
0 \lra\Omega_{{X}/T}
\lra(\ParAt_{\widetilde{X}/T}(\Pcal)(D))^{\vee}\lra(\SParEnd(\Pcal)(D))^{\vee}\lra 0. 
\end{equation}
The Cartan-Killing form $\kappa_{\gfrak}$ gives an  identification 
\begin{equation} \label{eqn:CK}
\nu_\gfrak^{-1}: (\SParEnd(\Pcal)(D))^{\vee}\isorightarrow
\ParEnd(\Pcal).
\end{equation}
Pulling  back the exact sequence in \eqref{eqn:quasiLiealgebra} by 
the above isomorphism, we get 
a quasi-Lie algebra ${}^{spar}\widetilde{\At}_{X/T}(\Pcal)$
fitting into the following exact sequence:
\begin{equation*}
\adjustbox{scale=0.9}{
	\begin{tikzcd}
	0 \arrow[r]& \Omega_{{X}/T} \arrow[r]& (p_*^{\Gamma}(\At_{\widetilde{X}/T}(\widehat\Pcal)(-\widehat{D}))(D))^{\vee} \arrow[r]& {(\SParEnd(\Pcal)(D))^{\vee}\cong(p_*^{\Gamma}(\ad(\widehat{\Pcal})(-\widehat{D}))(D))^{\vee}}\arrow[r]& 0\\
	0 \arrow[r]& \Omega_{{X}/T}\arrow[u, equal] \arrow[r]& 
	{}^{spar}\widetilde{\At}_{X/T}(\Pcal) \arrow[r]\arrow[u,"\cong"]& \ParEnd(\Pcal)\arrow[r]\arrow[u,"\cong"]& 0
	\end{tikzcd}}
\end{equation*}
which should be considered as a parabolic $G$-bundle 
analog of Ginzburg's dgla considered in \cite{Ginzburg}. Recall that we
have the  quasi-Lie algebra $\Sfrak^{-1}(\widehat{\Pcal})$ associated to a family $\widehat{\Pcal}$ of principal $G$-bundles satisfying: 
\begin{equation*}
\begin{tikzcd}
0 \arrow[r] & \Omega_{\widetilde{X}/T} \arrow[r]& \At_{\widetilde{X}/T}(\widehat{\Pcal})^{\vee} \arrow[r] & \ad(\widehat{\Pcal})^{\vee} \arrow[r] &0\\
0 \arrow[r] & \Omega_{\widetilde{X}/T}\arrow[r] \arrow[u,equal]&
\Sfrak^{-1}_{\widetilde{X}/T}(\widehat{\Pcal}) \arrow[u,"\cong"]\arrow[r] & \ad(\widehat{\Pcal}) \arrow[r] \arrow[u,"\cong"]&0.
\end{tikzcd}
\end{equation*}
Now since $\widehat{ \Pcal}$ is $\Gamma$-linearized, we conclude that all objects in the above exact
sequence are $\Gamma$-linearized.
We define ${}^{par}\Sfrak^{-1}_{{X}/T}(P):= p_*^{\Gamma}\Sfrak^{-1}_{\widetilde{X}/T}(\widehat{ \Pcal})$. 

Taking $\Gamma$-invariant pushforward of the bottom row, we get the following extension 
\begin{equation*}
0 \lra \Omega_{X/T} \lra {}^{par}\Sfrak^{-1}_{{X}/T}(\Pcal ) \lra \ParEnd(\Pcal)\lra 0.
\end{equation*}

We now have the following proposition: 

\begin{proposition}\label{prop:paraoblicGdgla}
	There is an isomorphism
	${}^{par}\Sfrak^{-1}_{{X}/T}(\Pcal)
	\isorightarrow {}^{spar}\widetilde{\At}_{X/T}(\Pcal)$ which induces identity maps on $\Omega_{X/T}$ and $\Par(\Pcal)$. 
\end{proposition}

\begin{proof} Recall that there is a natural nondegenerate pairing
	$\langle\,  ,\,  \rangle: \mathfrak{S}^{-1}_{\widetilde{X}/T}(\widehat\Pcal)\times \At_{\widetilde{X}/T}(\widehat \Pcal)
	\rightarrow \mathcal{O}_{\widetilde{X}}.$
	Tensoring it with $\mathcal{O}_{\widetilde{X}}(-\widehat{D})$, we get a $\mathcal{O}_{\widetilde{X}}
	(-\widehat{D})$-valued pairing
	$$
	\langle\,  ,\,  \rangle: \mathfrak{S}^{-1}_{\widetilde{X}/T}(\widehat\Pcal)\times
	\At_{\widetilde{X}/T}(\widehat \Pcal)(-\widehat{D})\lra\mathcal{O}_{\widetilde{X}}(-\widehat{D}).
	$$
	Taking invariant pushforward, we get the following nondegenerate pairing:
	$$p_*^{\Gamma}\langle\,  ,\,  \rangle:
	p_*^{\Gamma}\Sfrak^{-1}_{\widetilde{X}/T}(\widehat\Pcal)\times
	p_*^{\Gamma}(\At_{\widetilde{X}/T}(\widehat \Pcal)(-\widehat{D}))\lra p_*^{\Gamma}\mathcal{O}_{\widetilde{X}}(-\widehat{D}).$$
	This produces a duality between ${}^{par}\Sfrak^{-1}_{{X}/T}(\Pcal)$ and $p_*^{\Gamma}(\At_{\widetilde{X}/T}(\widehat{\Pcal})(-\widehat{D}))(D)$. This completes the proof of the proposition.
\end{proof}
\subsection{The general set of parabolic $G$-bundles}
Let $M^{\bm\tau,ss}_{G}$ (respectively, $M^{\bm\tau,rs}_{G}$) be the moduli space of semistable
(respectively, regularly stable) $(\Gamma, G)$ 
bundles on a curve $\widehat{C}$, and let $\phi: G \rightarrow \SL_r$ be a representation. We assume without loss of generality that $\widehat{M}_{G}^{rs}$ is non-empty. 
Note that for a semistable 
$(\Gamma,G)$, the underlying $G$-bundle is semistable (\cite{Biswasduke}, \cite{BBNtohuku}).
We also use the same 
notation for a relative family of $\Gamma$ covers $\widehat{\mathcal{C}}\rightarrow S$.
Consider the induced maps ${M}^{\bm\tau,ss}_{G}\to\widehat{M}^{ss}_{G} \stackrel{\phi}{\lra} 
\widehat{M}^{ss}_{\SL_r}$. Let $\widehat{M}^{rs}_{G}$ be the locus of regularly stable bundles on $\widehat{C}$. 
Then by \cite{Faltings:93}, we get that the complement of the regularly stable locus is at least two provided 
$g\geq 2$ and $G$ is any simple group different from
$\SL_2$ and $g\geq 3$ if the Lie algebra of $G$ has a $\mathfrak{sl}(2)$ factor. 
Now as in the $\SL_r$ case, let $Y^{\bm \tau 
	,rs}_{G}$ be the inverse image of $\widehat{M}^{rs}_G$ in $M_{G}^{\bm \tau,ss}$. Moreover, the complement of 
$Y^{\bm \tau,rs}_G$ has codimension at least two provided the genus
$g(\mathscr{C})$ of the orbifold curve $\mathscr{C}=[\widehat{C}/\Gamma]$
(cf.\ \cite[Lemma 8.3]{BMW1}) determined by $\bm \tau$ is at least three,
or  if the Lie algebra of $G$ has no $\mathfrak{sl}(2)$ or $\mathfrak{sp}(4)$ factor,  $g(\mathscr{C})\geq 2$.

As before let $\mathcal{L}_{\phi}$ be the pull-back of the determinant of
cohomology $\mathcal{L}$ to $M^{\bm\tau,ss}_{G}$.
Then by applying Proposition \ref{prop:paraoblicGdgla}, and Theorem \ref{thm:atiyahG}, we get
the following.

\begin{corollary} \label{cor:atiyah-det-ginzburg}
	There is a natural isomorphism of Atiyah algebras over the regularly stable locus $Y_{G}^{\bm \tau,rs}$:
	$$\frac{1}{m_{\phi}}\At_{M_{G}^{\bm \tau,rs}/S}(\mathcal{L}_{\phi})\isorightarrow
	R^1\pi_{n \ast}\left({}^{par}\Sfrak^{-1}_{\mathcal{C}\times_S M_{G}^{\bm \tau ,rs}/M_G^{\bm \tau, rs}}(\Pcal)\right).$$
	Since by assumption, the complement of $Y_{G}^{\bm \tau,rs}$ in $M_{G}^{\bm \tau,rs}$ is at least two, the above isomorphism extends over the entire space $M_{G}^{{\bm \tau},rs}$. 
\end{corollary}

\bibliographystyle{amsplain}
\bibliography{papers}

\end{document}